\documentclass[11pt, dvipsnames, reqno]{amsart}
\usepackage[pagewise]{lineno}
\usepackage[T1]{fontenc}
\usepackage{a4wide}
\usepackage{centernot}
\usepackage{amsfonts,latexsym,rawfonts,amsmath,amssymb,amsthm, mathrsfs, lscape, calligra}
\usepackage{yhmath,mathrsfs,amsthm,amsmath,amssymb,amsfonts, enumerate,lipsum, appendix,mathtools, bm}
\usepackage{bbold}
\usepackage[mathscr]{euscript}
\usepackage[normalem]{ulem} 
\usepackage{graphicx}
\usepackage{pgf,tikz}
\usepackage{tkz-euclide}
\usepackage{graphicx}
\usepackage{subcaption}
\usetikzlibrary{shapes.geometric}
\usetikzlibrary{arrows,hobby}
\usepackage{enumitem}
\usepackage[english]{babel}

\allowdisplaybreaks

\usepackage[sort,numbers]{natbib}
\usepackage[colorlinks=true,urlcolor=red, citecolor=red,linkcolor=blue,
linktocpage,pdfpagelabels, bookmarksnumbered,bookmarksopen]{hyperref}
\usepackage{doi}
\definecolor{reed}{rgb}{0.0, 0.0, 0.65}
\usepackage[margin=0.85in, heightrounded]{geometry}
\allowdisplaybreaks
\usepackage{orcidlink}
\usepackage{bigints}
 \usepackage{esint}
\newtheorem{theorem}{Theorem}[section]
\newtheorem{lemma}[theorem]{Lemma}
\newtheorem{proposition}[theorem]{Proposition}

\theoremstyle{definition}

\newtheorem{remark}[theorem]{Remark}
\numberwithin{equation}{section}
\usepackage[hyperpageref]{backref}

\newcommand*\rd{\mathbb{R}^d}

\newcommand*\N{\mathcal{N}}
\newcommand{\al} {\alpha}
\newcommand{\pa} {\partial}

\newcommand{\de} {\delta}

\newcommand{\la} {\lambda}

\newcommand{\Gr} {\nabla}
\newcommand{\no} {\nonumber}
\newcommand{\noi} {\noindent}
\newcommand{\var} {\varepsilon}
\newcommand{\ra} {\rightarrow}

\newcommand{\bee} {\begin{equation}}
\newcommand{\eee} {\end{equation}}
\newcommand{\bea} {\begin{eqnarray}}
\newcommand{\eea} {\end{eqnarray}}
\newcommand{\Bea} {\begin{eqnarray*}}
\newcommand{\Eea} {\end{eqnarray*}}

\def\d{\,{\rm d}}
\def\dx{\,{\rm d}x}
\def\dy{\,{\rm d}y}

\def\C{{\mathcal C}}
\def\D{{\mathcal D}}

\def\wps{{\mathcal W}}

\def\R{{\mathbb R}}
\def\N{{\mathbb N}}

\def\({{\Big(}}
\def\){{\Big)}}
\def\cc{{\C_c^\infty}}

\def\dx{\,{\rm d}x}
\def\dxy{\,{\rm d}x\,{\rm d}y}

\def\dxnyn{\,{\rm d}\bar{x}_n\,{\rm d}\bar{y}_n}
\def\dz{\,{\rm d}z}

\DeclarePairedDelimiter\abs{\lvert}{\rvert}
\DeclarePairedDelimiter\norm{\lVert}{\rVert}
\def\wps{{\mathcal{D}^{s,p}}}
\def\wos{{\mathcal{D}_0^{s,p}}(\Omega)}
\usepackage{orcidlink}

\usepackage{fancyhdr}

\title[Global compactness results]{Global compactness results for fractional $p$-Laplace Hardy Sobolev operator on a bounded domain}
\author[N. Biswas]{Nirjan Biswas$^1$\,\orcidlink{0000-0002-3528-8388}}
\address{\rm  $^1$ Department of Mathematics, Indian Institute of Science Education and Research Pune \\
Dr. Homi Bhabha Road, Pune 411008, India \vspace{0.2 cm}}
\begin{document}
\begin{abstract}
In this paper, we establish a Struwe type global compactness result for a class of nonlinear critical Hardy-Sobolev exponent problems driven by the fractional $p$-Laplace Hardy-Sobolev operator.
\end{abstract}
	\medskip
	
	\noindent
	\maketitle
\section{Introduction}

This paper aims to study the global compactness result for the fractional $p$-Laplace Hardy-Sobolev operator, in the spirit of the framework introduced in \cite{Smets-TAMS}. For $s\in (0,1), p \in (1, \infty)$, and $d>sp$, we consider the following critical problem driven by the fractional $p$-Laplace Hardy Sobolev operator on a bounded open set $\Omega$ containing origin:
\begin{equation}\label{MainEq}\tag{$\mathcal{P}_{\mu,a,\alpha}$}
			(-\Delta_p)^s u -\mu\dfrac{\abs{u}^{p-2}u}{|x|^{sp}} + a(x)\abs{u}^{p-2}u = \frac{|u|^{p^*_s(\al)-2}u}{|x|^{\al}} \;\mbox{ in }\,\Omega, \; u=0 \;\mbox{ in }\,\rd \setminus \Omega, 
\end{equation}
where $\mu >0$, $0\leq \al<sp$, $p^*_s(\al):=\frac{p(d-\al)}{d-sp}$ is the critical Hardy-Sobolev exponent, which coincides with the critical Sobolev exponent $p^*_s:=\frac{dp}{d-sp}$ when $\al=0$, and $a \in L^{\frac{d-\alpha}{sp-\alpha}}(\Omega)$ is the weight function. The fractional $p$-Laplace operator $(-\Delta _p)^s$ is defined as
\begin{equation*}
   (-\Delta_{p})^{s}u(x) = 2 \lim_{\var \rightarrow 0^{+}} \int_{\mathbb{R}^{d} \backslash B(x, \var)} \frac{|u(x) - u(y)|^{p-2}(u(x)-u(y))}{|x-y|^{d+sp}}\, \dy, \; \text{for}~x \in \mathbb{R}^{d},
\end{equation*} 
where $B(x, \var)$ denotes the ball of radius $\var$ with center at $x \in \rd$. For the solution space of \eqref{MainEq}, first we recall the fractional homogeneous space $\wps$, which is defined as the completion of $\mathcal{C}^{\infty}_{c}(\rd)$ under the Gagliardo seminorm
$$\norm{u}_{\wps}\coloneqq \left( \; \iint_{\rd \times \rd} \frac{|u(x)-u(y)|^p}{|x-y|^{d+sp}}\dxy \right)^{\frac{1}{p}}.$$
The space $\wps$ has the following characterization (see \cite[Theorem 3.1]{Brasco2019characterisation}):
\begin{equation*}
    \wps\coloneqq\left\{u\in L^{p^*_s}(\rd):\norm{u}_{\wps}<\infty\right\},
\end{equation*}
where $\norm{\cdot}_{\wps}$ is an equivalent norm in $\wps$ and it is a reflexive Banach space.
For details on $\wps$ and its associated embedding results, we refer to \cite{Brasco2019characterisation, DiPaVa} and the references therein. Recall the Hardy-Sobolev inequality (see \cite[Theorem 1.1]{FrSe}):
\begin{align}\label{HS}
    C(d,s,p) \int_{\rd} \frac{\abs{u}^p}{\abs{x}^{sp}} \dx \le \iint_{\rd \times \rd}\frac{|u(x)-u(y)|^p}{|x-y|^{d+sp}}\,\dx\dy, \; \forall \, u \in \wps.
\end{align}
Let $\mu_{d,s,p}$ be the best constant of \eqref{HS}, i.e., 
\begin{align*}
    \mu_{d,s,p} := \inf_{u \in \wps \setminus \{ 0\}} \displaystyle \frac{\norm{u}_{\wps}^p}{\displaystyle \int_{\rd} \frac{\abs{u}^p}{\abs{x}^{sp}} \dx}.
\end{align*}
The explicit expression of $\mu_{d,s,p}>0$ is derived in \cite[Theorem 1.1]{FrSe}. Also, from \cite[Theorem 1.1]{FrSe} it is known that the inequality \eqref{HS} is strict for every $u \in \wps \setminus \{0\}$. 
If $\mu<\mu_{d,s,p}$, then 
\begin{align}\label{equivalent-norm}
    \norm{u}_{\mu} := \left( \norm{u}^p_{\wps} - \mu \int_{\rd} \frac{|u|^p}{|x|^{sp}} \dx \right)^{\frac{1}{p}}, 
\end{align}
is an equivalent norm in $\wps$, i.e., there exists $C_{\text{eqiv}}>0$ such that $C_{\text{eqiv}}\norm{u}_{\wps}^p \le \norm{u}_{\mu}^p$, for every $u \in \wps$. Now we consider the following closed subspace of $\wps$:
\begin{align*}
    \wos := \left\{ u \in \wps : u=0 \text{ in } \rd \setminus \Omega \right\},
\end{align*}
as a solution space for \eqref{MainEq}. It is endowed with the norm $\norm{u}_{\wps}$, and an equivalent norm  $\norm{u}_{\mu}$.

For $0 <\mu<\mu_{d,s,p}$, we consider the following energy functional associated with \eqref{MainEq}:
\begin{align}\label{energy}
    I_{\mu,a, \alpha}(u) := &\frac{1}{p}\norm{u}_{\wps}^p - \frac{\mu}{p} \int_{\Omega} \frac{\abs{u}^p}{\abs{x}^{sp}} \dx + \frac{1}{p} \int_{\Omega} a(x) \abs{u}^p \dx \no \\
    &- \frac{1}{p^*_s(\al)} \int_{\Omega} \frac{|u|^{p^*_s(\al)}}{|x|^{\al}} \dx, \; \forall \, u \in \wos.
\end{align}
In view of \eqref{equivalent-norm}, the embedding $\wps \hookrightarrow L^{p^*_s(\al)}(\rd, \abs{x}^{-\al})$, and \eqref{int-welldefined}, $I_{\mu,a,\alpha}$ is well-defined. Moreover, $I_{\mu,a,\alpha} \in \mathcal{C}^1(\wos, \R)$. If $u \in \wos$ is a critical point of $I_{\mu,a,\alpha}$, then it satisfies the following identity: 
\begin{align}\label{weak1}
    &\iint_{\rd \times \rd}\frac{|u(x)-u(y)|^{p-2}(u(x)-u(y))(\phi(x)-\phi(y))}{|x-y|^{d+sp}} \dxy - \mu \int_{\Omega} \frac{\abs{u}^{p-2} u}{\abs{x}^{sp}} \phi \dx \no \\ 
    &+ \int_{\Omega} a(x) \abs{u}^{p-2} u \phi \dx =\int_{\Omega} \frac{\abs{u}^{p^*_s(\alpha)-2}u}{\abs{x}^{\alpha}} \phi \dx, \; \forall \, \phi \in \wos,
\end{align}
i.e., $u$ a weak solution of \eqref{MainEq}.
A sequence $\{ u_n \} \subset \wos$ is said to be a Palais-Smale (PS) sequence for $I_{\mu,a,\alpha}$ at level $\eta$, if $I_{\mu,a,\alpha}(u_n) \ra \eta$ in $\R$ and $I'_{\mu,a,\alpha}(u_n) \ra 0$ in $(\wos)'$ as $n \ra \infty$. The function $I_{\mu,a,\alpha}$ is said to satisfy (PS) condition at level $\eta$, if every (PS) sequence at level $\eta$ has a convergent subsequence. Observe that every (PS) sequence of $I_{\mu,a,\alpha}$ may not converge strongly due to the noncompactness of the embedding $\wos \hookrightarrow L^{p^*_s(\al)}(\Omega, \abs{x}^{-\al})$. Moreover, the weak limit of the (PS) sequence can be zero even if $\eta>0$. In this paper, we classify the (PS) sequence for the functional $I_{\mu,a,\alpha}$, and as an application, we find the existence of a positive weak solution to \eqref{MainEq}.

The classification of (PS) sequences was first studied by Struwe \cite{Struwe} for the following functional: 
\begin{align*}
    I_{\la}(u) = \frac{1}{2} \int_{\Omega} \abs{\Gr u}^2 - \frac{\la}{2} \int_{\Omega} u^2 - \frac{1}{2^*} \int_{\Omega} \abs{u}^{2^*},\; u \in \mathcal{D}_0^{1,2}(\Omega), 
\end{align*}
where $\la \in \R$, $\Omega$ is a smooth bounded domain in $\rd$ with $d>2$, $2^*=\frac{2d}{d-2}$ is the critical exponent, and $\mathcal{D}_0^{1,2}(\Omega)$ is the closure of $\cc(\Omega)$ with respect to $\norm{\Gr \cdot}_{L^2(\Omega)}$. Observe that every critical point of $I_{\la}$ weakly solves the Br\'ezis-Nirenberg problem 
\begin{equation}\label{brezis-nirernberg}
   -\Delta u = \la u + \abs{u}^{2^*-2}u \text{ in } \Omega; u=0 \text{ on } \pa \Omega. 
\end{equation} 
From \cite{BrNi}, it is known that below the level $\frac{1}{d}S^{\frac{d}{2}}$, where $S$ is the best constant of $\mathcal{D}_0^{1,2}(\Omega) \hookrightarrow L^{2^*}(\Omega)$, every (PS) sequence for $I_{\la}$ contains a convergent subsequence. This opens the question of classifying all the ranges where $I_{\la}$ fails to satisfy the (PS) condition. In \cite{Struwe}, Struwe answered this question by showing that if $\{ u_n \}$ is a (PS) sequence of $I_{\la}$ at level $c$, then there exist an integer $k \ge 0$, sequences $\{ x_n^i \}_n \subset \rd, \{ r_n^i \}_n \subset \R^+$, a set of functions $u \in \mathcal{D}_0^{1,2}(\Omega), \tilde{u}_i \in \mathcal{D}^{1,2}(\rd)$ for $1 \le i \le k$ (where $\mathcal{D}^{1,2}(\rd) := \{u \in L^{2^*}(\rd) : \abs{\Gr u} \in L^2(\rd)\}$), such that $u$ weakly solves \eqref{brezis-nirernberg}, $\tilde{u}_i$ weakly solves the purely critical problem on $\rd$, i.e., $-\Delta \tilde{u}_i = \abs{\tilde{u}_i}^{2^*-2}\tilde{u}_i$ in $\rd$ such that the following hold: 
\begin{align*}
    u_n = u + \sum_{i=1}^k \tilde{u}_i^{r_n^i, x_n^i} + o_n(1), \text{ in } \mathcal{D}^{1,2}(\rd),
\end{align*}
where $o_n(1) \ra 0$ as $n \ra \infty$, and
\begin{align*}
    \tilde{u}_i^{r, y}(x) \coloneqq  r^{-\frac{d
    -2}{2}} \tilde{u}_i \left( \frac{x-y}{r} \right), \text{ for } x,y \in \rd,\,r>0.
\end{align*}
Moreover, the energy level $c$ is distributed in the following manner:
\begin{align*}
    c = I_{\la}(u) + \sum_{i=1}^k I_{\infty}(\tilde{u}_i) + o_n(1), \text{ where } I_{\infty}(u) = \frac{1}{2} \int_{\rd} \abs{\Gr u}^2 - \frac{1}{2^*} \int_{\rd} \abs{u}^{2^*}, \; u \in \mathcal{D}^{1,2}(\rd).   
\end{align*}
This result is valuable to investigate the existence of ground states in nonlinear Schr\"{o}dinger equations, Yamabe-type equations, and various types of minimization problems. Following this work, several authors investigated the (PS) decomposition of the energy functional associated with both local and nonlocal operators on bounded domains. Notable contributions in this direction include \cite{Willem-2010, Gerard, Palatucci-Pisante-NA}. In \cite[Theorem 1.2]{Willem-2010}, Mercuri and Willem studied a similar (PS) decomposition of \eqref{MainEq} with $\mu=0, \alpha =0$, and $s=1$, i.e., for the following functional:
\begin{align*}
    I_p(u) \coloneqq \frac{1}{p} \int_{\Omega} \abs{\Gr u}^p + \frac{1}{p} \int_{\Omega} a(x) \abs{u}^p - \frac{\mu}{p^*} \int_{\Omega} \abs{u}^{p^*}, \; u \in \mathcal{D}_0^{1,p}(\Omega),
\end{align*}
where $\mu>0, d>p$, $a \in L^{\frac{d}{p}}(\Omega)$, $p^*=\frac{dp}{d-p}$ is the critical exponent, and $\mathcal{D}_0^{1,p}(\Omega)$ is the closure of $\cc(\Omega)$ with respect to $\norm{\Gr \cdot}_{L^p(\Omega)}$. In \cite[Theorem 1.1]{Brasco-2016}, Brasco, Squassina, and Yang studied the global compactness result for \eqref{MainEq} with $\mu=0$ and $\alpha=0$. They also studied the global compactness result for radially symmetric functions defined on a ball $B \subset \rd$.
For $p=2$ and $f \in (\mathcal{D}^{s,2})'$, Bhakta and Pucci in \cite[Proposition 2.1]{Bhakta-Pucci} classified the (PS) sequences associated with the following energy functional:
\begin{equation}
   I_{a,f}(u) \coloneqq \frac{1}{2} \norm{u}_{\mathcal{D}^{s,2}}^2 - \frac{1}{2^*_s} \int_{\rd} a(x) \abs{u}^{2^*_s} -  \prescript{}{(\mathcal{D}^{s,2})'}{\langle}f,u{\rangle}_{\mathcal{D}^{s,2}},\;  u \in \mathcal{D}^{s,2},
\end{equation}
where $0<a\in L^{\infty}(\rd)$, $a(x) \ra 1$ as $\abs{x} \ra \infty$. More precisely, they established that if $\{ u_n \}$ is a (PS) sequence of $I_{a,f}$ at level $c$, then there exist an integer $k \ge 0$, sequences $\{ x_n^i \}_n \subset \rd, \{ r_n^i \}_n \subset \R^+$, a set of functions $u, \tilde{u}_i \in \mathcal{D}^{s,2}$ for $1 \le i \le k$, such that $r_n^i \ra 0, \text{ and either } x_n^i \ra x^i \in \rd \text{ or } \abs{x_n^i} \ra \infty, \text{ for } 1 \le i \le k$,
$u$ weakly solves $(-\Delta)^s u = a(x) \abs{u}^{2^*_s -2}u + f$ in $\rd$, and $\tilde{u}_i \not\equiv 0$ weakly solves the corresponding homogeneous equation $(-\Delta)^s u = a(x^i) \abs{u}^{2^*_s -2}u$ in $\rd$ such that following hold:
\begin{align*}
    u_n = u + \sum_{i=1}^k a(x^i)^{-\frac{d-2s}{4s}} \tilde{u}_i^{r_n^i, x_n^i} + o_n(1), \text{ in } \mathcal{D}^{s,2},
\end{align*}
where $o_n(1) \ra 0$ as $n \ra \infty$, and 
\begin{align*}
    \tilde{u}_i^{r, y}(x) \coloneqq  r^{-\frac{d
    -2s}{2}} \tilde{u}_i \left( \frac{x-y}{r} \right), \text{ for } x,y \in \rd,\,r>0.
\end{align*}
In this case, the energy level $c$ is distributed in the following manner:
\begin{align*}
    c = I_{a,f}(u) + \sum_{i=1}^k a(x^i)^{-\frac{d-2s}{2s}} I_{1,0}(\tilde{u}_i) + o_n(1).   
\end{align*}
Observe that, by the uniqueness of the positive solution of 
\begin{align*}
    (-\Delta)^s u = u^{2^*_s-1} \text{ in } \rd,\, u \in \D^{s,2},
\end{align*}
for each $i$, $\tilde{u}_i$ is a nonlocal Aubin-Talenti bubble (up to translation and dilation). In \cite[Proposition 2.1]{Bhakta-Pucci}, the authors established the following bubble interaction (which is motivated by the works of Palatucci and Pisante (see \cite{GiAd2014, Palatucci-Pisante-NA})):
\begin{align}\label{bub-1}
\left| \log \left( \frac{r_n^i}{r_n^j} \right) \right| + \left| \frac{x_n^i - x_n^j}{r_n^i}  \right| \rightarrow \infty, \text{ for } 1 \le i \le k.
\end{align}
The (PS) decomposition in the context of systems of equations has been investigated by Peng, Peng, and Wang in \cite{PPW} for $s=1$ and $p=2$, and by Bhakta, Chakraborty, Miyagaki, and Pucci in \cite{MoSoMiPa} for $s \in (0,1)$ and $p=2$. The second work has been recently extended by Biswas and Chakraborty \cite{NS2025} for $s \in (0,1)$ and $p \in (1, \infty)$. In \cite{NS2025}, the authors observed that even for $p \neq 2$, a bubble interaction of the same type as in \eqref{bub-1} still arises.

Smets in \cite{Smets-TAMS} studied the following nonlinear Schr\"{o}dinger equation:
\begin{align}\label{Hardy-Sobolev-II}
-\Delta u - \mu\frac{u}{\abs{x}^2} = K(x) \abs{u}^{2^*-2}u \text{ in } \rd, u \in \mathcal{D}^{s,2},    
\end{align}
where $\mu>0$ and $K \in L^{\infty}(\rd)$. The author observed that, in the presence of a Hardy potential $\abs{x}^{-2}$, noncompactness arises due to concentration occurring through two distinct profiles: the usual Aubin-Talenti bubble, and a second, genuinely different profile solving the limit Hardy-Sobolev equation ($K\equiv1$); (see also \cite{Bhak-San} for the local case involving Hardy-Sobolev-Maz'ya type equations).  
In \cite[Theorem 2.1]{BCP}, Bhakta, Chakraborty, and Pucci further extended this global compactness result (see \cite[Theorem 3.1]{Smets-TAMS}) for $s \in (0,1)$. They obtained a complete classification of (PS) sequences for the critical Hardy-Sobolev functional associated with the fractional Hardy-Sobolev exponent $2^*_s(\al)=\frac{2(d-\al)}{d-2s}$, extending the Palatucci-Pisante and Bhakta-Pucci decompositions \cite{Palatucci-Pisante-NA, Bhakta-Pucci} to the singular setting for $\al\in(0,2s)$. 

In this paper, we extend the study of \cite{Brasco-2016} by incorporating the fractional $p$-Laplace Hardy-Sobolev operator and Hardy potential $|x|^{-\al}$ with $\al \in (0,sp)$. For $0 \le \mu<\mu_{d,s,p}$, we consider the following quantity: 
\begin{align*}
  S_{\mu} := S_{\mu}(d,p,s, \mu, \al) = \inf_{u \in \wps \setminus \{ 0\}} \frac{\norm{u}_{\wps}^p - \mu \displaystyle \int_{\rd} \frac{\abs{u}^p}{\abs{x}^{sp}}\dx }{\displaystyle \left(\int_{\rd} \frac{\abs{u}^{p^*_s(\alpha)}}{\abs{x}^{\alpha}} \dx\right)^{\frac{p}{p_s^*(\alpha)}}}.
\end{align*}
For brevity, we denote $S_{\mu}$ as $S:= S(d,p,s)$ when $\mu=0$ and $\alpha=0$. In this case, it is known from \cite[Theorem 1.1]{Brasco2016} that $S>0$ is attained by an extremum which is positive, radially symmetric, radially decreasing at the origin, and has a certain decay at infinity.  
For $\mu >0$, Shen in \cite[Theorem 1.1]{Shen24} (when $\al=0$), and Assunção, Silva, and Miyagaki in \cite[Theorem 1.2]{ARSJ2020} (when $\al>0$), proved that $S_{\mu}>0$ is attained by a non-negative extremum which is again positive, radially symmetric and radially decreasing with respect to the origin. These extrema (up to a multiplication of a normalized constant) satisfy the following equations weakly: 
\begin{equation}\label{limit-problem-intro}
\begin{aligned}
   &\text{I}: \quad   (-\Delta_p)^s u = |u|^{p^*_s-2}u \;\mbox{ in }\,\mathbb{R}^d, \; u \in \wps, \\
   &\text{II}: \quad (-\Delta_p)^s u -\mu\dfrac{\abs{u}^{p-2}u}{|x|^{sp}}=\frac{|u|^{p^*_s(\alpha)-2}u}{\abs{x}^{\alpha}} \;\mbox{ in }\,\mathbb{R}^d, \; u \in \wps. 
\end{aligned}
\end{equation}
The uniqueness of these extrema is not known. Nevertheless, from this discussion, we note that the solution sets for \eqref{limit-problem-intro} are non-empty. Now, we are in a position to state our main result. We would like to point out that this result is new even in the local case $s=1$.

The following theorem classifies the (PS) sequence for the energy functional $I_{\mu,a,\alpha}$ defined in \eqref{energy}. 

\begin{theorem}\label{PS-decomposition}
  Let $s\in (0,1), p \in (1, \infty), \mu \in (0,\mu_{d,s,p})$, and $\al \in [0,sp)$.  Let $\Omega$ be a bounded open set containing origin and $a \in L^{\frac{d-\alpha}{sp-\alpha}}(\Omega)$. Let $\{u_n\}$ be a (PS) sequence for $I_{\mu,a,\alpha}$ at level $\eta$. Then there exists a subsequence (still denoted by $\{ u_n \}$) for which the following hold:
  
  \noi there exist $n_1, n_2 \in \N \cup \{0\}$, sequences $\{ r_n^i \}_n \subset \R^+$ for $1\le i\le n_1$, and $\{ x_n^j \}_n \subset \rd, \{ R_n^j \}_n \subset \R^+$ for $1\le j\le n_2$, functions $u, \tilde{u}_i, \tilde{U}_j$ (where $1\le i\le n_1$ and $1\le j \le n_2$) such that $u$ weakly satisfies \eqref{MainEq} without sign assumptions, $\tilde{u}_i$ weakly satisfies   
  \begin{align*}
    (-\Delta_p)^s \tilde{u}_i -\mu\dfrac{\abs{\tilde{u}_i}^{p-2}\tilde{u}_i}{|x|^{sp}}=\frac{|\tilde{u}_i|^{p^*_s(\al)-2}\tilde{u}_i}{\abs{x}^{\al}} \;\mbox{ in }\,\mathbb{R}^d, \; \tilde{u}_i \in \wps \setminus \{0\},
\end{align*}
and $\tilde{U}_j$ weakly satisfies
\begin{align*}
    (-\Delta_p)^s \tilde{U}_j=|\tilde{U}_j|^{p^*_s-2}\tilde{U}_j \;\mbox{ in }\,\mathbb{R}^d, \; \tilde{U}_j \in \wps \setminus \{0\},
\end{align*}
such that 
\begin{align*}
    &(a) \textbf{ Profile decomposition: }\text{the (PS) sequence decomposes as}\\
    & \qquad \qquad u_n = u + \sum_{i=1}^{n_1} C_{r_n^i}(\tilde{u}_i) + \sum_{j=1}^{n_2} C_{x_n^j, R_n^j}(\tilde{U}_j) + o_n(1), \\
    &(b) \textbf{ Norm decomposition: } \norm{u_n}_{\wps}^p = \norm{u}_{\wps}^p + \sum_{i=1}^{n_1} \norm{\tilde{u}_i}_{\wps}^p + \sum_{j=1}^{n_2} \norm{\tilde{U}_j}_{\wps}^p + o_n(1), \\
    &(c) \textbf{ Energy decoupling: }  \eta = I_{\mu, a, \alpha}(u) + \sum_{i=1}^{n_1} I_{\mu, 0, \alpha}(\tilde{u}_i) + \sum_{j=1}^{n_2} I_{0,0,0}(\tilde{U}_j), \\ 
    &(d) \textbf{ Behavior of parameters: } r_n^i \ra 0, R_n^j \ra 0, x_n^j \ra x^j \in \rd \text{ or } \abs{x_n^j} \ra \infty, \text{ and } \\
    &\qquad \qquad \frac{R_n^j}{\abs{x_n^j}} \ra 0, \text{ for } 1\le i\le n_1, 1\le j \le n_2,\\
    &(e) \textbf{ Asymptotic orthogonality: } \left| \log \left( \frac{r_n^i}{r_n^j} \right) \right| \rightarrow \infty, \text{ for } i \neq j, 1 \le i, j \le n_1, \text{ and } \\
    & \qquad \qquad \left| \log \left( \frac{R_n^i}{R_n^j} \right) \right| + \left| \frac{x_n^i - x_n^j}{R_n^i}  \right| \rightarrow \infty, \text{ for } i \neq j, 1 \le i, j \le n_2,
\end{align*}
where $o_n(1) \ra 0$ as $n \ra \infty$ in $\wps$, and $C_{r_n^i}(\tilde{u}_i), C_{x_n^j, R_n^j}\tilde U_j$ are defined as 
\begin{align*}
    C_{r_n^i}(\tilde{u}_i) \coloneqq (r_n^i)^{-\frac{d-sp}{p}}\tilde u_i\left(\frac{x}{r_n^i}\right), \text{ and } C_{x_n^j, R_n^j}\tilde U_j(x)\coloneqq (R_n^j)^{-\frac{d-sp}{p}}\tilde U_j\left(\frac{x-x_n^j}{R_n^j}\right).
\end{align*} 
In the case $n_1=0$ and $n_2=0$, the above expression holds without $\tilde{u}_i, r_n^i, \tilde{U}_j, R_n^j$, and $x_n^j$. Furthermore, if $\al >0$, then the same conclusion holds with $n_2=0$.
\end{theorem}

\begin{remark}
(a) In the (PS) decomposition of $I_{0,a,0}$ (see \cite[Theorem 1.1]{Brasco-2016}), the following limiting equations appear: 
\begin{align}\label{limit-problem}
\text{I:}\; (-\Delta_p)^s u  = \abs{u}^{p^*_s-2}u \;\mbox{ in }\, \rd, \quad \text{II:}\; (-\Delta_p)^s u  = \abs{u}^{p^*_s-2}u \;\mbox{ in }\,\mathcal{H}, \; u=0 \;\mbox{ in }\,\rd \setminus \mathcal{H},
\end{align}
where $\mathcal{H} \subset \rd$ is a upper-half plane. Note that \eqref{limit-problem} is invariant under both translation and scaling. For this reason, in the Levy concentration function (constructed in \cite[Step 2, p. 406]{Brasco-2016}), the sequences $\{x_n\} \subset \rd$ and $\{r_n\} \subset \R^+$ arise, where $r_n \ra 0$ and $\frac{\text{dist}(x_n, \pa \Omega)}{r_n} \ra \{0, \infty \}$. Depending on the values of the second limit, $\tilde{u}_i$ weakly satisfies any one of \eqref{limit-problem}. In \cite[Theorem 1.1]{Brasco-2016}, the non-existence of any non-trivial weak solution to \eqref{limit-problem}-(II) is assumed, which immediately infers that $\tilde{u}_i$ has to satisfy \eqref{limit-problem}-(I). On the other hand, when $\mu>0$ and $\al>0$, the presence of the Hardy potentials in \eqref{MainEq} ensures that the limiting equation is only invariant under scaling. In this situation, $\tilde{u}_i$ satisfies the limiting equation \eqref{limit-problem-intro}-(II) only on $\rd$.

\noi (b) In contrast with \cite[Theorem 1.1]{Brasco-2016}, note that due to the presence of the fractional $p$-Laplace Hardy-Sobolev operator, when $\al=0$, two distinct types of bubbles arise in the (PS) decomposition: one weakly solves \eqref{limit-problem-intro}-(I) and the other weakly solves \eqref{limit-problem-intro}-(II). On the other hand, when $\al>0$, only one type of bubbles appears in the (PS) decomposition.
\end{remark}

The rest of the paper is organised as follows. In the next section, we present several technical results that are essential for the proof of Theorem \ref{PS-decomposition}. Section 3 is devoted to the proof of Theorem \ref{PS-decomposition}.

\noi \textbf{Notation:} 
We use the following notation. 

\noi (a) $ \displaystyle \mathcal{A}(u,v) \coloneqq \iint_{\rd \times \rd} \frac{\abs{u(x) - u(y)}^{p-2} (u(x) - u(y)) (v(x) - v(y))}{\abs{x-y}^{d+sp}} \dxy.$ (b) For a set $A \subset \rd$, $|A|$ denotes the Lebesgue measure of $A$. (c) $\chi$ denotes the characteristic function. (d) $C$ denotes a generic positive constant. (e) $\norm{\cdot}_{p} := \norm{\cdot}_{L^p(\rd)}$. 

\section{Preliminaries}
This section presents several technical results that will be used in the subsequent analysis. We begin by recalling the classical Brézis-Lieb lemma and some of its consequences. 

\begin{lemma}\label{BL}
    Let $1<q< \infty$. Let $\{ f_n \} \subset L^{q}(\rd)$ be a bounded sequence such that $f_n(x) \ra f(x)$ a.e. $x \in \rd$. Then the following hold: 
    \begin{enumerate}
        \item[\rm{(i)}] $\norm{f_n}_{L^q(\rd)}^q - \norm{f_n - f}_{L^q(\rd)}^q + o_n(1) = \norm{f}_{L^q(\rd)}^q.$
        \item[\rm{(ii)}] Consider the function $J_q$ defined as $J_q(t)=\abs{t}^{q-2}t$. Then 
        \begin{align*}
           J_q(f_n) - J_q(f_n-f) = J_q(f) + o_n(1) \text{ in } L^{q'}(\rd).
        \end{align*}
    \end{enumerate}
\end{lemma}
\begin{proof}
Proof of (i) follows from \cite{Br-Li}, and proof of (ii) follows from \cite[Lemma 3.2]{Willem-2010}. 
\end{proof}
The above lemma leads to the following convergence.  

\begin{lemma}\label{convergence-BL}
    Let $\{u_n\}$ weakly converge to $u$ in $\wps$ and $u_n(x) \ra u(x)$ a.e. $x \in \rd$. Then up to a subsequence, the following hold
   \begin{enumerate}
       \item[\rm{(i)}] $\norm{u_n}_{\wps}^p - \norm{u_n - u}_{\wps}^p = \norm{u}^p_{\wps} + o_n(1)$. 
       \item[\rm{(ii)}] For $g \in L_{loc}^1(\rd)$ with $\int_{\rd} g(x) \abs{u}^p < \infty$, we have  
       \begin{align*}
           \displaystyle \int_{\rd} g(x) \abs{u_n}^{p} \dx -  \int_{\rd} g(x) \abs{u_n-u}^p \dx= \int_{\rd} g(x) \abs{u}^{p} \dx + o_n(1).
       \end{align*}
       \item[\rm{(iii)}] For $\tilde{\al} \in [0,sp]$, we have $$\displaystyle \int_{\rd} \frac{\abs{u_n}^{p^*_s(\tilde{\al})}}{\abs{x}^{\tilde{\al}}} \dx -  \int_{\rd} \frac{\abs{u_n-u}^{p^*_s(\tilde{\al})}}{\abs{x}^{\tilde{\al}}} \dx= \int_{\rd} \frac{\abs{u}^{p^*_s(\tilde{\al})}}{\abs{x}^{\tilde{\al}}} \dx + o_n(1).$$
       \item[\rm{(iv)}] Consider the function $J_p$ defined as $J_p(t)=\abs{t}^{p-2}t$. Then 
       \begin{align*}
           \frac{J_p(u_n(x) - u_n(y))}{\abs{x-y}^{\frac{d+sp}{p'}}} - \frac{J_p\left((u_n(x)-u(x)) - (u_n(y)-u(y))\right)}{\abs{x-y}^{\frac{d+sp}{p'}}} = \frac{J_p(u(x) - u(y))}{\abs{x-y}^{\frac{d+sp}{p'}}} + o_n(1),
       \end{align*}
       in $L^{p'}(\mathbb{R}^{2d})$.
       \item[\rm{(v)}] For $\tilde{\al} \in [0,sp]$, consider the function $J_{p^*_s(\tilde{\al})}$ defined as $J_{p^*_s(\tilde{\al})}(t)=\abs{t}^{p^*_s(\tilde{\al})-2}t$. Then 
       \begin{align*}
           \frac{J_{p^*_s(\tilde{\al})}(u_n(x))}{\abs{x}^{\frac{\tilde{\al}}{(p^*_s(\tilde{\al}))'}}} - \frac{J_{p^*_s(\tilde{\al})}(u_n(x)-u(x))}{\abs{x}^{\frac{\tilde{\al}}{(p^*_s(\tilde{\al}))'}}} = \frac{J_{p^*_s(\tilde{\al})}(u(x))}{\abs{x}^{\frac{\tilde{\al}}{(p^*_s(\tilde{\al}))'}}} + o_n(1),
       \end{align*}
       in $L^{(p^*_s(\tilde{\al}))'}(\rd)$.
   \end{enumerate}
    \end{lemma}
The following lemma states the convergence of some integrals. For proof, we refer to \cite[Lemma 2.5]{NS2025}.

\begin{lemma}\label{convergence-integrals}
Let $\{ u_n \}$ weakly converge to $u$ in $\wps$. 
    \begin{enumerate}
    \item[\rm{(i)}] Let $g \in L_{loc}^1(\rd)$ with $\int_{\rd} g(x) \abs{u}^p < \infty$. Then up to a subsequence  
        \begin{align*}
            & \lim_{n \ra \infty} \int_{\rd} g(x) \abs{u_n(x)}^{p -2} u_n(x) \phi(x) \dx = \int_{\rd} g(x) \abs{u(x)}^{p -2} u(x) \phi(x) \dx, \text{ and } \\
            & \lim_{n \ra \infty} \int_{\rd} g(x) \abs{\phi(x)}^{p -2} \phi(x) u_n(x) \dx = \int_{\rd} g(x) \abs{\phi(x)}^{p -2} \phi(x) u(x) \dx,
        \end{align*}
        for every $\phi \in \wps$.
    \item[\rm{(ii)}] Let $\tilde{\al} \in [0,sp]$. Then up to a subsequence
    \begin{align*}
        & \lim_{n \ra \infty} \int_{\rd} \frac{\abs{u_n(x)}^{p^*_s(\tilde{\al}) -2} u_n(x)}{\abs{x}^{\tilde{\al}}} \phi(x) \dx = \int_{\rd} \frac{\abs{u(x)}^{p^*_s(\tilde{\al}) -2} u(x)}{\abs{x}^{\tilde{\al}}} \phi(x) \dx, \\ & \lim_{n \ra \infty} \int_{\rd} \frac{\abs{\phi(x)}^{p^*_s(\tilde{\al}) -2} \phi(x)}{\abs{x}^{\tilde{\al}}} u_n(x) \dx = \int_{\rd} \frac{\abs{\phi(x)}^{p^*_s(\tilde{\al}) -2} \phi(x)}{\abs{x}^{\tilde{\al}}} u(x) \dx.
    \end{align*}
    for every $\phi \in \wps$.
    \item[\rm{(iii)}] Then up to a subsequence
    \begin{align*}
    \mathcal{A}(u_n,\phi) \ra \mathcal{A}(u,\phi), \text{ and } \mathcal{A}(\phi,u_n) \ra \mathcal{A}(\phi,u),
    \end{align*}
    for every $\phi \in \wps$.
    \end{enumerate}
\end{lemma}

\begin{remark}\label{convergence-integrals-remark}
In particular, all the convergences in Lemma \ref{convergence-BL} and Lemma \ref{convergence-integrals} hold for a sequence $\{ u_n \}$ with $u_n \rightharpoonup u$ in $\wos$.  
\end{remark}

\begin{lemma}\label{compact-embedding}
   Let $\al \in [0,sp)$ and $a \in L^{\frac{d-\alpha}{sp-\alpha}}(\Omega)$. Then the following embedding into the weighted Lebesgue space: $$\wos \hookrightarrow L^p(a, \Omega)$$ is compact. 
\end{lemma}

\begin{proof}
Using H\"{o}lder's inequality with the conjugate pair $(\frac{d-\alpha}{sp-\alpha}, \frac{d-\alpha}{d-sp})$, and the embedding $\wos \hookrightarrow L^{p^*_s(\alpha)}(\Omega)$, observe that 
\begin{align}\label{int-welldefined}
    \int_{\Omega} |a(x)| |u|^p \dx \le \norm{a}_{L^{\frac{d-\alpha}{sp-\alpha}}(\Omega)} \norm{u}^p_{p^*_s(\alpha)} \le\norm{a}_{L^{\frac{d-\alpha}{sp-\alpha}}(\Omega)} \norm{u}_{\wps}^{p}, \; \forall \, u \in \wos. 
\end{align}
Hence $\wos \hookrightarrow L^p(a, \Omega)$. Let $\var>0$ be given. By the density of $\cc(\Omega)$ in $L^{\frac{d-\alpha}{sp-\alpha}}(\Omega)$, there exists $a_{\var} \in \cc(\Omega)$ such that $\norm{a-a_{\var}}_{L^{\frac{d-\alpha}{sp-\alpha}}(\Omega)} < \var$. Let $K:= \text{supp}(a_{\var})$.  If $u_n \rightharpoonup u$ in $\wos$, then  
    \begin{align*}
        \int_{\Omega} \abs{a} \abs{u_n -u}^p \dx \le \int_{\Omega} \abs{a-a_{\var}} \abs{u_n -u}^p \dx + \norm{a_{\var}}_{L^{\infty}(K)} \int_{\Omega} \abs{u_n-u}^p \dx = o_n(1),
    \end{align*}
    since $\wos \hookrightarrow L^p(\Omega)$ compactly, and 
    \begin{align*}
        \int_{\Omega} \abs{a-a_{\var}} \abs{u_n -u}^p \dx \le C(d,p,s)\norm{a-a_{\var}}_{L^{\frac{d-\alpha}{sp-\alpha}}(\Omega)} \norm{u_n -u}^{p}_{\wps} \le C \var.
    \end{align*}
    Hence, the embedding $\wos \hookrightarrow L^p(a, \Omega)$ is compact. 
\end{proof}

Let $\mathcal{T} \subset \mathcal{U}(\wps)$ be the following class of isometric operators induced by the composition of translations and dilations on $\rd$:
\begin{align*}
   \mathcal{T}\coloneqq \left\{C_{y,\lambda} \in \mathcal{U}(\wps): \la \in (0, \infty) \text{ and } C_{y,\lambda}u(x)\coloneqq \lambda^{-\frac{d-sp}{p}}u \left(\frac{x-y}{\lambda}\right), \; \forall \, u\in\wps;\,y\in\rd\right\}. 
\end{align*}
A change of variable shows that
\begin{equation}\label{eq:invariance}
\norm{C_{y,\la}u}_{\wps}= \norm{u}_{\wps}, \text{ and } \int_{\rd}\abs{C_{y,\la}u}^{p^*_s}\dx = \int_{\rd}\abs{u}^{p^*_s}\dx, \; \forall\, \la \in (0, \infty),\, y \in \rd.
\end{equation}
Let $(y_1,\delta), (y_2, \lambda) \in \rd\rtimes (0,\infty)$. Then for every $u, v \in \wps$,
\begin{align}\label{cov-1}
\mathcal{A} \left( C_{y_1,\delta}u, C_{y_2, \lambda} v \right) = \mathcal{A}\left( u, C_{\frac{y_2-y_1}{\delta}, \frac{\lambda}{\delta}} v \right) = \mathcal{A}(C_{\frac{y_1-y_2}{\la},\frac{\delta}{\la}}u,v).
\end{align}

The following proposition states that if a sequence in the group $G = \rd\rtimes (0,\infty)$ sends every element in $\wps$ to $0$ under the action $\mathcal{A}$, then the sequence must go to infinity with respect to the metric $d$ of $G$, defined as
\begin{align*}
     \displaystyle d\left((y,\lambda),(w,\delta)\right) \coloneqq \left| \log\left(\frac{\lambda}{\delta}\right) \right| + |y-w|.
\end{align*}
\begin{proposition}[{\cite[Proposition 2.8]{NS2025}}]\label{weak-bub}
Let $\{(a_n,\delta_n)\},\,\{(y_n,\lambda_n)\}\subset G$ be such that 
\begin{align*}
\mathcal{A}(C_{a_n,\delta_n}u,C_{y_n,\lambda_n}v)\to 0, \text{ for every }u,v\in\wps.
\end{align*}
Then $\left| \log\left(\frac{\delta_n}{\lambda_n}\right) \right| + \left| \frac{a_n-y_n}{\lambda_n} \right| \to\infty,$ as $n \ra \infty$.
\end{proposition}

Let $\mathcal{D}\subset \mathcal{U}(\wps)$ be a class of isometric operators induced by dilations on $\rd$:
\begin{align*}
  \mathcal{D} \coloneqq \left\{C_{\lambda}\in \mathcal{U}(\wps) : \lambda\in (0,\infty) \text{ and }C_{\lambda}u(x)=\lambda^{-\frac{d-sp}{p}}u \left( \frac{x}{\lambda} \right) \right\}.\
\end{align*}

Observe that, for every $u, v \in \wps$, 
\begin{align}\label{cov-2}
\mathcal{A} \left( C_{\delta}u, C_{\lambda} v \right) = \mathcal{A}\left( u, C_{\frac{\lambda}{\delta}} v \right) = \mathcal{A}(C_{\frac{\delta}{\la}}u,v), \; \forall \,  \delta, \lambda \in (0, \infty). 
\end{align}

Using a similar set of arguments as in \cite[Proposition 2.8]{NS2025}, we also have the following convergence.

\begin{proposition}\label{weak-bub_1}
Let $\{\delta_n\},\,\{\lambda_n\}\subset (0, \infty)$ be such that 
\begin{align*}
\mathcal{A}(C_{\delta_n}u,C_{\lambda_n}v)\to 0, \text{ for every }u,v \in\wps.
\end{align*}
Then $\left| \log\left(\frac{\delta_n}{\lambda_n}\right) \right| \to\infty,$ as $n \ra \infty$.
\end{proposition}

\section{Global compactness results} 

\noi \textbf{Proof of Theorem \ref{PS-decomposition}:} Since $\{u_n\} \subset \wps$ is a (PS) sequence of $I_{\mu,a,\alpha}$ at level $\eta$, we have 
\begin{align}\label{PSD-0}
 \norm{u_n}_{\mu}^p + \int_{\Omega} a(x) \abs{u_n}^p \dx - \frac{p}{p^*_s(\al)} \int_{\Omega} \frac{|u_n|^{p^*_s(\al)}}{|x|^{\al}} \dx  =  pI_{\mu,a,\alpha}(u_n) = p\eta + o_n(1), 
\end{align}
and 
\begin{align}\label{PSD-1}
    pI_{\mu,a,\alpha}(u_n) - \prescript{}{(\wps)'}{\langle} I'_{\mu,a,\alpha}(u_n),(u_n){\rangle}_{\wps} \le C + o_n(1) \norm{u_n}_{\mu}.
\end{align}
Now 
\begin{align*}
    \text{ L.H.S. of \eqref{PSD-1} } \ge \left(1- \frac{p}{p^*_s(\al)} \right) \int_{\Omega} \frac{|u_n|^{p^*_s(\al)}}{|x|^{\al}} \dx.
\end{align*}
Hence, in view of \eqref{PSD-0} and \eqref{PSD-1}, we see that 
\begin{align}\label{PSD-1.1}
    \int_{\Omega} \frac{|u_n|^{p^*_s(\al)}}{|x|^{\al}} \dx \le C\left(1+ \norm{u_n}_{\wps}\right).
\end{align}
Further, the H\"{o}lder's inequality with the conjugate pair $(\frac{d-\alpha}{d-sp},\frac{d-\alpha}{sp-\alpha})$ yields
\begin{align*}
    \left|\int_{\Omega} a(x) \abs{u_n}^p \dx\right| & \le \int_{\Omega} \frac{\abs{u_n}^p}{\abs{x}^{\frac{\alpha(d-sp)}{d-\alpha}}} \abs{a(x)} \abs{x}^{\frac{\alpha(d-sp)}{d-\alpha}}  \dx  \\
    &\le \norm{\abs{x}^{\frac{\alpha(d-sp)}{d-\alpha}}}_{L^{\infty}(\Omega)} \left( \int_{\Omega} \frac{\abs{u(x)}^{p^*_s(\alpha)}}{\abs{x}^{\alpha}} \dx \right)^{\frac{d-sp}{d-\alpha}} \left( \int_{\Omega} \abs{a(x)}^{\frac{d-\alpha}{sp-\alpha}} \dx \right)^{\frac{sp-\alpha}{d-\alpha}}\\
    & \le C(d,s,p, \alpha) \norm{a}_{L^{\frac{d-\alpha}{sp-\alpha}}(\Omega)} \left(1 + \norm{u_n}_{\mu}^{\frac{p}{p^*_s(\alpha)}} \right).
\end{align*}
Hence, in view of \eqref{PSD-0}, $\{ u_n \}$ is a bounded sequence on $\wps$. By the reflexivity of $\wps$, let $\{ u_n \}$ weakly converge to $\tilde{u}$ in $\wps$ (up to a subsequence). Since $I'_{\mu,a,\alpha}(u_n) \ra 0$ in $(\wps)'$, for every $\phi \in \wps$ we have
\begin{align*}
    \mathcal{A}(u_n , \phi) - \mu \int_{\rd} \frac{\abs{u_n}^{p-2} u_n}{\abs{x}^{sp}} \phi \dx +\int_{\Omega} a(x) \abs{u_n}^{p-2} \phi \dx =\int_{\Omega} \frac{\abs{u}^{p^*_s(\alpha)-2}u}{\abs{x}^{\alpha}} \phi \dx, \; \forall \, \phi \in \wos.
\end{align*}
Taking the limit as $n \ra \infty$ in the above identity and using Lemma \ref{convergence-integrals}, we see that $\tilde{u}\in \wos$ satisfies \eqref{weak1} weakly. We divide the rest of the proof into several steps. 

\noi \textbf{Step 1:} In this step, we claim that $\{ u_n - \tilde{u}\}$ is a (PS) sequence for $I_{\mu,0,\alpha}$ at level $\eta -I_{\mu,a,\alpha}(\tilde{u})$. Set $\tilde{u}_n = u_n - \tilde{u}$. Using Lemma \ref{convergence-BL} and $\tilde{u}_n \rightharpoonup 0$ in $\wos$, we get
\begin{align*}
    I_{\mu,0,\alpha}(\tilde{u}_n) & = \frac{1}{p} \norm{\tilde{u}_n}_{\wps}^p - \frac{\mu}{p} \int_{\Omega} \frac{\abs{\tilde{u}_n}^p}{\abs{x}^{sp}} \dx - \frac{1}{p^*_s(\alpha)} \int_{\Omega} \frac{\abs{\tilde{u}_n}^{p^*_s(\alpha)}}{\abs{x}^{\alpha}} \dx \\
    & = \frac{1}{p} \left( \norm{u_n}_{\wps}^p - \norm{\tilde{u}}_{\wps}^p \right) - \frac{\mu}{p} \left( \int_{\Omega} \frac{\abs{u_n}^p - \abs{\tilde{u}}^p}{\abs{x}^{sp}} \dx \right) - \frac{1}{p^*_s(\alpha)} \left( \int_{\Omega}  \frac{\abs{u_n}^{p^*_s(\alpha)} - \abs{\tilde{u}}^{p^*_s(\alpha)}}{\abs{x}^{\alpha}} \dx \right) \\
    & + \int_{\Omega} a(x) \left( \abs{u_n}^p - \abs{\tilde{u}}^p \right) \dx  + o_n(1) = I_{\mu,a,\alpha}(u_n) - I_{\mu,a,\alpha}(\tilde{u}) + o_n(1).
\end{align*}
The second identity follows using Lemma \ref{compact-embedding}.
Hence $I_{\mu,0,\alpha}(\tilde{u}_n) \ra \eta - I_{\mu,a,\alpha}(\tilde{u})$ as $n \ra \infty$. Further, for $\phi \in \wos$, using Remark \ref{convergence-integrals-remark} we have 
\begin{align*}
    \prescript{}{(\wps)'}{\langle} I_{\mu,0,\alpha}'(\tilde{u}_n), \phi {\rangle}_{\wps} = \mathcal{A}(\tilde{u}_n , \phi) - \mu \int_{\Omega} \frac{\abs{ \tilde{u}_n }^{p-2} \tilde{u}_n}{\abs{x}^{sp}} \phi\,\dx - \int_{\Omega} \frac{\abs{\tilde{u}_n}^{p^*_s(\alpha) -2} \tilde{u}_n}{\abs{x}^{\alpha}} \phi \,\dx = o_n(1).
\end{align*}
Thus, the claim holds.

\noi \textbf{Step 2:} Suppose $u_n \ra \tilde{u}$ in $\wos$. From the continuity of $I_{\mu,a,\alpha}$, we get $\eta = I_{\mu,a,\alpha}(\tilde{u})$, and Theorem \ref{PS-decomposition} holds for $n_1=n_2=0$. So, from now onward we assume that $u_n \not\ra \tilde{u}$ in $\wos$. 
In view of Step 1, $\prescript{}{(\wps)'}{\langle} I_{\mu,0,\alpha}'(\tilde{u}_n), \tilde{u}_n {\rangle}_{\wps} \ra 0$, which implies 
\begin{align}\label{del-0}
    0 < c \le C_{\text{eqiv}} \norm{\tilde{u}_n}_{\wps}^p \le  \norm{\tilde{u}_n}^p_{\wps} - \mu \int_{\Omega} \frac{\abs{\tilde{u}_n}^p}{\abs{x}^{sp}} \dx = \int_{\Omega} \frac{\abs{\tilde{u}_n}^{p^*_s(\alpha)}}{\abs{x}^{\alpha}} \dx + o_n(1).
\end{align}
In view of \eqref{del-0}, there exists $\delta_1>0$ such that 
\begin{align*}
    \inf_{n \in \N} \int_{\rd}  \frac{\abs{\tilde{u}_n}^{p^*_s(\alpha)}}{\abs{x}^{\alpha}} \dx = \delta_1.
\end{align*}
We take $0< \de< \de_1$ and consider the Levy concentration function
\begin{align*}
    Q_n(r) \coloneqq \int_{B(0,r)} \frac{\abs{\tilde{u}_n}^{p^*_s(\alpha)}}{\abs{x}^{\alpha}} \dx.
\end{align*}
Observe that $Q_n(0) =0$ and $Q_n(\infty) > \delta$. Further, $Q_n$ is continuous on $\R^+$ (see \cite[Lemma 3.1]{Brasco-2016}). Hence, there exists $\{ r_n \} \subset \R^+$ such that  
\begin{align}\label{int-1}
Q_n(r_n) = \int_{B(0, r_n)} \frac{\abs{\tilde{u}_n}^{p^*_s(\alpha)}}{\abs{x}^{\alpha}} \dx = \delta. 
\end{align}
If $r_n \ge \text{diam}(\Omega)$, then 
\begin{align*}
  \delta =  Q_n(r_n) = \int_{\Omega} \frac{\abs{\tilde{u}_n}^{p^*_s(\alpha)}}{\abs{x}^{\alpha}} \dx = \int_{\rd} \frac{\abs{\tilde{u}_n}^{p^*_s(\alpha)}}{\abs{x}^{\alpha}} \dx > \delta,
\end{align*}
a contradiction. Therefore, $r_n < \text{diam}(\Omega)$ for every $n \in \N$, i.e., the sequence $\{ r_n \}$ is bounded. Let $r_n \ra r_0$ in $\R^+$.  
We set 
\begin{align*}
\hat{u}_n(z) \coloneqq r_n^{\frac{d-sp}{p}} \tilde{u}_n(r_n z ), \text{ for } z \in \frac{\Omega}{r_n}. 
\end{align*}
Using the change of variable and \eqref{int-1}, 
\begin{align}\label{int-1.5}
\int_{B(0, 1)} \frac{\abs{\hat{u}_n}^{p^*_s(\alpha)}}{\abs{x}^{\alpha}} \dx = \delta. 
 \end{align}  
By observing the fact that $\norm{\tilde{u}_n}_{\wps} = \norm{\hat{u}_n}_{\wps}$, the sequence $\{ \hat{u}_n\}$ is bounded on $\wps$.  By the reflexivity of $\wps$, let $\hat{u}_n \rightharpoonup \hat{u}$ in $\wps$. Now, the following steps are based on several cases, depending on whether $\alpha>0$ or $\alpha=0$.

\noi \textbf{Step 3:} In this step, we first assume $\al>0$ and show that $\hat{u} \neq 0$. On the contrary, suppose $\al>0$ and $\hat{u} = 0$. Consider $\phi \in \cc(B(0,1))$ with $0 \le \phi \le 1$. 
Set
\begin{align*}
    \phi_n(z) \coloneqq \phi\left(\frac{z}{r_n}\right) \tilde{u}_n(z), \text{ for } z \in \rd. 
\end{align*}
Note that $\text{supp}(\phi_n) \subset B(0, r_n)$. Since $\{ \tilde{u}_n \}$ is a (PS) sequence of $I_{\mu,0,\al}$, we have  
\begin{align}\label{int-2}
  \mathcal{A}(\tilde{u}_n, \phi_n) = \mu \int_{\rd} \frac{\abs{\tilde{u}_n}^{p-2} \tilde{u}_n}{\abs{x}^{sp}} \phi_n \dx + \int_{\rd} \frac{\abs{\tilde{u}_n}^{p^*_s(\al)-2} \tilde{u}_n}{\abs{x}^{\al}} \phi_n \dx + o_n(1).
\end{align}
Now we estimate $\mathcal{A}(\tilde{u}_n, \phi_n)$. Using the change of variable $\bar{x}_n = \frac{x}{r_n}, \bar{y}_n = \frac{y}{r_n}$, we write 
\begin{align*}
    &\mathcal{A}(\tilde{u}_n, \phi_n) \\
    &= \iint_{\rd \times \rd} \frac{\abs{\tilde{u}_n(x) - \tilde{u}_n(y)}^{p-2}(\tilde{u}_n(x) - \tilde{u}_n(y)) \left( \phi(\frac{x}{r_n}) \tilde{u}_n(x) -  \phi(\frac{y}{r_n}) \tilde{u}_n(y) \right)}{\abs{x-y}^{d+sp}} \dxy \\
    &= r_n^{d-sp} \iint_{\rd \times \rd} \frac{\abs{\tilde{u}_n(r_n \bar{x}_n) - \tilde{u}_n(r_n \bar{y}_n)}^{p-2}(\tilde{u}_n(r_n \bar{x}_n) - \tilde{u}_n(r_n \bar{y}_n))}{\abs{\bar{x}_n-\bar{y}_n}^{d+sp}} \\
    & \quad \quad \quad \quad \left( \phi(\bar{x}_n)  \tilde{u}_n(r_n \bar{x}_n) -  \phi(\bar{y}_n) \tilde{u}_n(r_n \bar{y}_n) \right) \dxnyn \\
    & = \iint_{\rd \times \rd} \frac{\abs{\hat{u}_n(\bar{x}_n) - \hat{u}_n(\bar{y}_n)}^{p-2}(\hat{u}_n(\bar{x}_n) - \hat{u}_n(\bar{y}_n)) \left( \phi(\bar{x}_n) \hat{u}_n(\bar{x}_n) - \phi(\bar{y}_n) \hat{u}_n(\bar{y}_n) \right)}{\abs{\bar{x}_n-\bar{y}_n}^{d+sp}} \dxnyn.
\end{align*}
Applying the H\"{o}lder's inequality with the conjugate pair $(p,p')$,
\begin{align}\label{PS-decompo-1.1}
    \left| \mathcal{A}(\tilde{u}_n, \phi_n) \right| \le \norm{\hat{u}_n}_{\wps}^{p-1} \left( \iint_{\rd \times \rd} \frac{\abs{\phi(x)\hat{u}_n(x) - \phi(y)\hat{u}_n(y)}^p}{\abs{x-y}^{d+sp}} \dxy \right)^{\frac{1}{p}}.
\end{align}
Now we proceed to estimate the right-hand side integral of \eqref{PS-decompo-1.1}. We split 
\begin{align*}
    &\iint_{\rd \times \rd} \frac{\abs{\phi(x)\hat{u}_n(x) - \phi(y)\hat{u}_n(y)}^p}{\abs{x-y}^{d+sp}} \dxy \\
    & = \left( \iint_{B(0,1) \times B(0,1)} + 2 \iint_{B(0,1) \times B(0,1)^c} + \iint_{B(0,1)^c \times B(0,1)^c} \right) \frac{\abs{\phi(x)\hat{u}_n(x) - \phi(y)\hat{u}_n(y)}^p}{\abs{x-y}^{d+sp}} \dxy \\
    &:= I_1 + I_2 + I_3.
\end{align*}
Clearly $I_3 = 0$ as $\text{supp}(\phi) \subset B(0,1)$. We now show that $I_1$ is finite. For that
\begin{align*}
    &\iint_{B(0,1) \times B(0,1)} \frac{\abs{\phi(x)\hat{u}_n(x) - \phi(y)\hat{u}_n(y)}^p}{\abs{x-y}^{d+sp}} \dxy \\
    &\le 2^{p-1} \iint_{B(0,1) \times B(0,1)} \left( \abs{\hat{u}_n(x)}^p  \frac{\abs{\phi(x) - \phi(y)}^p}{\abs{x-y}^{d+sp}} + \abs{\phi(y)}^p \frac{\abs{ \hat{u}_n(x)- \hat{u}_n(y) }^p}{\abs{x-y}^{d+sp}} \right) \dxy,
\end{align*}
where 
\begin{align*}
    \iint_{B(0,1) \times B(0,1)} \abs{\phi(y)}^p \frac{\abs{ \hat{u}_n(x)- \hat{u}_n(y) }^p}{\abs{x-y}^{d+sp}} \dxy \le \norm{\phi}_{L^{\infty}(\rd)}^p \norm{ \hat{u}_n }_{\wps}^p \le C.
\end{align*}
Moreover, using $\abs{\phi(x) - \phi(y)} \le C\abs{x-y}$, we see that 
\begin{align*}
    \iint_{B(0,1) \times B(0,1)} \abs{\hat{u}_n(x)}^p  \frac{\abs{\phi(x) - \phi(y)}^p}{\abs{x-y}^{d+sp}} \dxy & \le C^p \iint_{B(0,1) \times B(0,1)}  \frac{\abs{\hat{u}_n(x)}^p}{\abs{x-y}^{d+sp-p}} \dxy \\
    & \le C^p \int_{B(0,1)} \left( \int_{B(0,2)} \frac{\dz}{\abs{z}^{d+sp-p}} \right) \abs{\hat{u}_n(x)}^p \dx \le C.
\end{align*}
This proves the finiteness of the integral. Now, we show $I_2 = o_n(1)$. Observe that
\begin{align*}
    I_2 = \iint_{B(0,1) \times B(0,1)^c} \frac{\abs{\phi(x)\hat{u}_n(x)}^p}{\abs{x-y}^{d+sp}} \dxy \le \norm{\phi}_{L^{\infty}(\rd)}^p \iint_{B(0,1) \times B(0,1)^c} \frac{\abs{\hat{u}_n(x)}^p}{\abs{y-x}^{d+sp}} \dxy,
\end{align*}
where using the change of variable, we estimate the last integral as
\begin{align*}
    \int_{B(0,1)} \abs{\hat{u}_n(x)}^p  \left( \int_{\abs{z}>1} \frac{\dz}{\abs{z}^{d+sp}} \right) \dx \le C(d,s,p) \int_{B(0,1)} \abs{\hat{u}_n(x)}^p \dx = o_n(1), 
\end{align*}
where the first inequality holds as 
\begin{align*}
    \int_{\{\abs{z}>1\}} \frac{\dz}{\abs{z}^{d+sp}} \le C, \text{ and } \int_{B(0,1)} \abs{\hat{u}_n(x)}^p \dx =o_n(1),
\end{align*}
using the compact embedding $\wps \hookrightarrow L^p_{loc}(\rd)$ and $\hat{u}=0$. Hence $I_2=o_n(1)$, and 
\begin{align}\label{finiteness}
    \iint_{\rd \times \rd} \frac{\abs{\phi(x)\hat{u}_n(x) - \phi(y)\hat{u}_n(y)}^p}{\abs{x-y}^{d+sp}} \dxy \le C,
\end{align}
for some $C>0$. Next, we show that $I_1 = o_n(1)$. From the compact embedding of $\wps \hookrightarrow L^{p}_{loc}(\rd)$ and $\hat{u}=0$, we have $\hat{u}_n(x) \ra 0$ pointwise a.e. $x \in B(0,1)$. This implies $\abs{\phi(x)\hat{u}_n(x) - \phi(y)\hat{u}_n(y)} \ra 0$ pointwise a.e. $x,y \in B(0,1)$. Define
\begin{align*}
    U_n(x,y) := \frac{(\phi(x)\hat{u}_n(x) - \phi(y)\hat{u}_n(y))}{\abs{x-y}^{\frac{d+sp}{p}}}, \; \text{ for } x,y \in \rd. 
\end{align*}
In view of \eqref{finiteness}, 
\begin{align*}
    \norm{U_n}_{L^p(\rd \times \rd)} \le C, \; \forall \, n \in \mathbb{N}. 
\end{align*}
Therefore, applying Vitali's convergence theorem, we conclude that
\begin{align*}
    \lim_{n \ra \infty} \iint_{B(0,1) \times B(0,1)} \frac{\abs{\phi(x)\hat{u}_n(x) - \phi(y)\hat{u}_n(y)}^p}{\abs{x-y}^{d+sp}} \dxy = \lim_{n \ra \infty} \norm{U_n}_{L^p(B(0,1) \times B(0,1))} =  0.
\end{align*}
The above convergence yields $I_1 = o_n(1)$.
Accumulating all the estimates, we get 
\begin{align}\label{zero norm}
 \norm{\phi \hat{u}_n}_{\wps} = o_n(1), \text{ whenever } \hat{u}=0. 
\end{align}
From \eqref{PS-decompo-1.1}, $ \mathcal{A}(\tilde{u}_n, \phi_n) = o_n(1)$. 
Now we show that 
\begin{align*}
    \int_{\Omega} \frac{\abs{\tilde{u}_n}^{p-2} \tilde{u}_n}{\abs{x}^{sp}} \phi_n \dx = o_n(1). 
\end{align*}
Using the change of variable and \eqref{HS}, we write 
\begin{align*}
    \int_{\rd} \frac{\abs{\tilde{u}_n}^{p-2} \tilde{u}_n}{\abs{x}^{sp}} \phi_n \dx = \int_{\rd} \frac{\abs{\hat{u}_n}^{p}}{\abs{x}^{sp}} \phi \dx = \int_{B(0,1)} \frac{\abs{\hat{u}_n}^{p}}{\abs{x}^{sp}} \phi \dx \le C \norm{\phi}_{L^{\infty}(\rd)} \norm{\hat{u}_n }^p_{\wps} \le C.
\end{align*}
Further, using the compact embedding $\wps \hookrightarrow L^p_{loc}(\rd)$ and $\hat{u}=0$, we see that $\frac{\abs{\hat{u}_n}^{p}}{\abs{x}^{sp}} \phi(x) \ra 0$ a.e. in $B(0,1)$. Hence, again using Vitali's convergence theorem, 
\begin{align*}
    \lim_{n \ra \infty} \int_{B(0,1)} \frac{\abs{\hat{u}_n}^{p}}{\abs{x}^{sp}} \phi \dx=0.
\end{align*}
Hence, in view of  \eqref{int-2}, we have 
\begin{align*}
o_n(1) = \int_{\rd} \frac{\abs{\tilde{u}_n}^{p^*_s(\al)-2} \tilde{u}_n}{\abs{x}^{\al}} \phi_n \dx = \int_{\rd} \frac{\abs{\hat{u}_n}^{p^*_s(\al)}}{\abs{x}^{\al}} \phi \dx,   
\end{align*}
where the last identity holds using the change of variable. Since $\phi \in \cc(B(0,1))$ is arbitrary,  for any $r\in (0,1)$ we can choose $\phi \equiv 1$ on $B_r$. Therefore,
\begin{align*}
 o_n(1) =  \int_{B_r}  \frac{\abs{\hat{u}_n}^{p^*_s(\al)}}{\abs{x}^{\al}} \dx,\text{ for any }0<r<1,
\end{align*}
which contradicts \eqref{int-1.5}. Thus, we conclude $\hat{u} \neq 0$.

For $\al=0$, we distinguish two cases: $\hat{u} =0$, and $\hat{u} \neq 0$. To be concise, in the remainder of this step, we consider $\al \in [0,sp)$ and $\hat{u} \neq 0$.

Suppose $r_0 >0$. Since $\hat{u} \neq 0$, we can choose $R>>1$ large enough so that $\norm{ \hat{u}}_{L^p(B(0,R))} >0$. Now using the compact embedding of $\wps \hookrightarrow L_{loc}^p(\rd)$ and applying the change of variable, we see that
\begin{align}\label{limit-1}
    0 < \norm{ \hat{u}}_{L^p(B(0,R))} = \norm{ \hat{u}_n }_{L^p(B(0,R))} +o_n(1) = r_n^{-s} \norm{\tilde{u}_n}_{L^p(B(0,r_nR)} + o_n(1).
\end{align}
Further, since $r_n \ra r_0$, there exists $R_1>0$ such that $B(0,r_nR) \subset B(0,R_1)$. Now again using the compact embedding of $\wps \hookrightarrow L_{loc}^p(\rd)$, 
\begin{align*}
   \lim_{n \ra \infty} r_n^{-s} \norm{\tilde{u}_n}_{L^p\left(B(0,r_nR)\right)} \le r_0^{-s} \lim_{n \ra \infty} \norm{\tilde{u}_n}_{L^p\left(B(0,R_1)\right)} = 0,
\end{align*}
which contradicts \eqref{limit-1}. Therefore, $r_n \ra 0$ as $n \ra \infty$. As a consequence, $|\rd \setminus \frac{\Omega}{r_n}| \ra 0$ as $n \ra \infty$. 

Now, we show that the non-zero weak limit $\hat{u}$ weakly solves the following limiting equation
\begin{align}\label{limiting problem-1}
    (-\Delta_p)^s u -\mu\dfrac{\abs{u}^{p-2}u}{|x|^{sp}}=\frac{|u|^{p^*_s(\al)-2}u}{\abs{x}^{\al}} \;\mbox{ in }\,\mathbb{R}^d. 
\end{align}
Take $\phi \in \wps$.  From Step 2, since $\hat{u}_n \rightharpoonup \hat{u}$, using Lemma \ref{convergence-integrals}-(iii), we get $\mathcal{A}(\hat{u}_n, \phi) \ra  \mathcal{A}(\hat{u}, \phi)$ as $n \ra \infty$. For $n \in \mathbb{N}$, we set 
\begin{align*}
    \tilde{\phi}_n(z) = r_n^{- \frac{d-sp}{p}} \phi \left(\frac{z}{r_n} \right), \text{ for } z \in \rd.
\end{align*}
Note that $\norm{\tilde{\phi}_n}_{\wps} = \norm{\phi}_{\wps}$. Next, using the change of variable $\overline{x}_n = r_n x,\,\overline{y}_n = r_n y$, 
\begin{align}\label{invariant-1}
    & \mathcal{A}(\hat{u}_n, \phi) = r_n^{\frac{d-sp}{p'}} \iint_{\rd \times \rd} \frac{\abs{\tilde{u}_n(r_n x) - \tilde{u}_n(r_n y)}^{p-2} (\tilde{u}_n(r_n x) - \tilde{u}_n(r_n y)) (\phi(x) - \phi(y))}{\abs{x-y}^{d+sp}}\dx\dy \no \\
    & = r_n^{-\frac{d-sp}{p}} \iint_{\rd \times \rd} \frac{\abs{\tilde{u}_n(r_n x) - \tilde{u}_n(r_n y)}^{p-2} (\tilde{u}_n(r_n x) - \tilde{u}_n(r_n y))(\phi(x) - \phi(y))}{\abs{r_n x-r_n y}^{d+sp}}  \dx\dy  \no \\
    & = r_n^{- \frac{d-sp}{p}} \iint_{\rd \times \rd} \frac{\abs{\tilde{u}_n(\overline{x}_n) - \tilde{u}_n(\overline{y}_n)}^{p-2} (\tilde{u}_n(\overline{x}_n) - \tilde{u}_n(\overline{y}_n)) \left(\phi \left(\frac{\overline{x}_n}{r_n} \right) - \phi \left(\frac{\overline{y}_n}{r_n} \right) \right)}{\abs{\overline{x}_n - \overline{y}_n}^{d+sp}} \d \overline{x}_n \d \overline{y}_n \no \\
    &= \mathcal{A}(\tilde{u}_n, \tilde{\phi}_n).
\end{align}
Similarly, for $\tilde{\alpha} \in [0, sp]$, we also have  
\begin{align}\label{invariant-2}
    \int_{\rd} \frac{\abs{\hat{u}_n}^{p^*_s(\tilde{\alpha})-2} \hat{u}_n}{\abs{x}^{\tilde{\alpha}}} \phi \dx = \int_{\rd} \frac{\abs{\tilde{u}_n}^{p^*_s(\tilde{\alpha})-2} \tilde{u}_n}{\abs{x}^{\tilde{\alpha}}} \tilde{\phi}_n \dx.  
\end{align}
Now using $\prescript{}{(\wps)'}{\langle} I_{\mu,0,\alpha}'(\tilde{u}_n), \tilde{\phi}_n {\rangle}_{\wps} \ra 0$, \eqref{invariant-1}, and \eqref{invariant-2}, we see that 
\begin{align}\label{limit}
    &\mathcal{A}(\hat{u}_n, \phi) - \mu \int_{\rd} \frac{\abs{\hat{u}_n}^{p-2} \hat{u}_n}{\abs{x}^{sp}} \phi \dx = \mathcal{A}(\tilde{u}_n, \tilde{\phi}_n) - \mu \int_{\rd} \frac{\abs{\tilde{u}_n}^{p-2} \tilde{u}_n}{\abs{x}^{sp}} \tilde{\phi}_n \dx \no \\
    &=\int_{\rd} a(x) \abs{\tilde{u}_n}^{p-2} \tilde{u}_n \tilde{\phi}_n \dx + \int_{\rd} \frac{\abs{\tilde{u}_n}^{p^*_s(\alpha)-2}\tilde{u}_n}{\abs{x}^{\alpha}} \tilde{\phi}_n \dx + o_n(1) \no \\
    &=\int_{\rd} \frac{\abs{\hat{u}_n}^{p^*_s(\alpha)-2} \hat{u}_n}{\abs{x}^{\alpha}} \phi \dx + o_n(1),
\end{align}
where the last line is obtained again by using the change of variable and the fact that 
\begin{align*}
    \left| \int_{\rd} a(x) \abs{\tilde{u}_n}^{p -2} \tilde{u}_n \tilde{\phi}_n \dx \right| & \leq \left( \int_{\Omega} \abs{a(x)} \abs{\tilde{u}_n}^p \dx \right)^{\frac{p-1}{p}} \left( \int_{\Omega} \abs{a(x)} \abs{\tilde{\phi}_n}^p \dx \right)^{\frac{1}{p}} \no \\
    & \le C \norm{a}_{L^{\frac{d-\alpha}{sp-\alpha}}(\Omega)} \norm{\phi}_{\wps} \left( \int_{\Omega} \abs{a(x)} \abs{\tilde{u}_n}^p \dx \right)^{\frac{p-1}{p}} = o_n(1),
\end{align*}
where $o_n(1)$ comes using the compact embedding of $\wos \hookrightarrow L^p(a, \Omega)$ (Lemma \ref{compact-embedding}) and the fact that $\{\tilde{u}_n\}$ weakly goes to zero in $\wos$. Now taking $n \ra \infty$ in \eqref{limit}, and applying Lemma \ref{convergence-integrals}-(ii), we see that $\hat{u}$ weakly solves \eqref{limiting problem-1}. 

Next, we set 
\begin{align*}
    w_n(z) = \tilde{u}_n(z) - r_n^{-\frac{d-sp}{p}} \hat{u} \left( \frac{z}{r_n} \right) \text{ and } \tilde{w}_n(z) = r_n^{\frac{d-sp}{p}} w_n(r_n z), \text{ for } z \in \rd.
\end{align*}
Note that  
\begin{align*}
    \norm{w_n}_{\wps} = \norm{\tilde{w}_n}_{\wps} \text{ and } \int_{\rd} \frac{\abs{w_n}^{p^*_s(\tilde{\alpha})}}{\abs{x}^{\tilde{\alpha}}} \dx = \int_{\rd} \frac{\abs{\tilde{w}_n}^{p^*_s(\tilde{\alpha})}}{\abs{x}^{\tilde{\alpha}}} \dx, \text{ for } \tilde{\alpha} \in [0,sp]. 
\end{align*}
Observe that $\tilde{w}_n = \hat{u}_n - \hat{u}$. Hence the norm invariance gives $\norm{ w_n }_{\wps} = \norm{ \tilde{w}_n }_{\wps} = \norm{ \hat{u}_n - \hat{u} }_{\wps}$. Applying Lemma \ref{convergence-BL}, we see that 
\begin{align*}
    \norm{ w_n }_{\wps}^p = \norm{ \hat{u}_n }_{\wps}^p - \norm{ \hat{u} }_{\wps}^p + o_n(1) = \norm{ \tilde{u}_n }_{\wps}^p - \norm{ \hat{u} }_{\wps}^p + o_n(1).
\end{align*}
We show that $\{w_n\}$ is a (PS) sequence of $I_{\mu,0,\alpha}$ at level $\eta - I_{\mu,a,\alpha}(\tilde{u}) - I_{\mu,0,\alpha}(\hat{u})$. 
Applying Lemma \ref{convergence-BL}, and the fact that $I_{\mu,0,\alpha}(\tilde{u}_n) = \eta -I_{\mu,a,\alpha}(\tilde{u}) + o_n(1)$, we see that 
\begin{align*}
    I_{\mu,0,\alpha} (w_n) & = \frac{1}{p}\norm{w_n}_{\wps}^p - \frac{\mu}{p} \int_{\rd}\frac{\abs{w_n}^p}{\abs{x}^{sp}} \dx - \frac{1}{p^*_s(\alpha)} \int_{\rd} \frac{\abs{w_n}^{p^*_s(\alpha)}}{\abs{x}^{\alpha}} \dx \\
    & = \frac{1}{p} \left( \norm{\hat{u}_n}_{\wps}^p - \norm{\hat{u}}_{\wps}^p \right) - \frac{\mu}{p} \left( \int_{\rd}\frac{\abs{\hat{u}_n}^p}{\abs{x}^{sp}} \dx - \int_{\rd}\frac{\abs{\hat{u}}^p}{\abs{x}^{sp}} \dx \right) \\
    & - \frac{1}{p^*_s(\alpha)} \left( \int_{\rd} \frac{\abs{\hat{u}_n}^{p^*_s(\alpha)}}{\abs{x}^{\alpha}} \dx - \int_{\rd} \frac{\abs{\hat{u}}^{p^*_s(\alpha)}}{\abs{x}^{\alpha}} \dx \right) + o_n(1) \\
    & = \frac{1}{p} \left( \norm{\tilde{u}_n}_{\wps}^p - \norm{\hat{u}}_{\wps}^p \right) - \frac{\mu}{p} \left( \int_{\rd}\frac{\abs{\tilde{u}_n}^p}{\abs{x}^{sp}} \dx - \int_{\rd}\frac{\abs{\hat{u}}^p}{\abs{x}^{sp}} \dx \right) \\
    & - \frac{1}{p^*_s(\alpha)} \left( \int_{\rd} \frac{\abs{\tilde{u}_n}^{p^*_s(\alpha)}}{\abs{x}^{\alpha}} \dx - \int_{\rd} \frac{\abs{\hat{u}}^{p^*_s(\alpha)}}{\abs{x}^{\alpha}} \dx \right) + o_n(1) \\
    &= I_{\mu,0,\alpha}(\tilde{u}_n) - I_{\mu,0,\alpha}(\hat{u}) + o_n(1) = \eta - I_{\mu,a,\alpha}(\tilde{u}) - I_{\mu,0,\alpha}(\hat{u}) + o_n(1).
\end{align*}
Next, we show $\prescript{}{(\wps)'}{\langle} I_{\mu,0,\alpha}'(w_n), \phi {\rangle}_{\wps} \ra 0$ for every $\phi \in \wps$. Using the boundedness of $\{w_n\}$ on $\wps$ and using the embedding $\wps \hookrightarrow L^{p^*_s(\alpha)}(\rd, \abs{x}^{- \alpha})$, we obtain the following
\begin{align*}
    \left| \prescript{}{(\wps)^*}{\langle} I_{\mu,0,\alpha}'(w_n), \phi {\rangle}_{\wps} \right| \le C\norm{\phi}_{\wps}, \text{ for every } \phi \in \wps. 
\end{align*}
Hence using the density of $\cc (\rd)$ in $\wps$, and using the uniform boundedness principle, it is enough to show $\prescript{}{(\wps)'}{\langle} I_{\mu,0,\alpha}'(w_n), \phi {\rangle}_{\wps} \ra 0$ for every $\phi \in \cc (\rd)$. For $\phi \in \cc (\rd)$, we define 
\begin{align*}
    \hat{\phi}_n(z) =  r_n^{\frac{d-sp}{p}} \phi(r_n z), \text{ for } z \in \rd.
\end{align*}
Since $\norm{ \hat{\phi}_n }_{\wps} = \norm{ \phi }_{\wps}$, the sequence $\{ \hat{\phi}_n \}$ is bounded in $\wps$, and up to a subsequence $\hat{\phi}_n \rightharpoonup u_1$ in $\wps$. Since $r_n\to 0$, $\hat{\phi}_n\to 0$ uniformly in $\R^{d}$ and $\{ \hat{\phi}_n\}$ is bounded in $\wps$, it has a weak limit (up to a subsequence) in $\wps$ which must coincide with $0$. Therefore, $u_1=0$ a.e. in $\rd$. Now using the change of variable, 
\begin{align}\label{ii-2}
    \prescript{}{(\wps)'}{\langle} I_{\mu,0,\alpha}'(w_n), \phi {\rangle}_{\wps} & = \mathcal{A}(w_n, \phi) - \mu \int_{\rd} \frac{\abs{w_n}^{p-2} w_n}{\abs{x}^{sp}} \phi \dx - \int_{\rd} \frac{\abs{w_n}^{p^*_s(\alpha)-2} w_n}{\abs{x}^{\alpha}} \phi \dx \no \\
    & = \mathcal{A}(\tilde{w}_n, \hat{\phi}_n) - \mu \int_{\rd} \frac{\abs{\tilde{w}_n}^{p-2} \tilde{w}_n}{\abs{x}^{sp}} \hat{\phi}_n \dx - \int_{\rd} \frac{\abs{\tilde{w}_n}^{p^*_s(\alpha)-2} \tilde{w}_n}{\abs{x}^{\alpha}} \hat{\phi}_n \dx.
\end{align}
Using Lemma \ref{convergence-BL}-(iv) and using H\"{o}lder's inequality with the conjugate pair $(p, p')$ and further using $\norm{\hat{\phi}_n}_{\wps} = \norm{\phi}_{\wps}$ we get $ \mathcal{A}(\tilde{w}_n, \hat{\phi}_n) -  \mathcal{A}(\hat{u}_n, \hat{\phi}_n) + \mathcal{A}(\hat{u}, \hat{\phi}_n) = o_n(1)$. Further, the change of variable yields $\mathcal{A} (\tilde{w}_n, \hat{\phi}_n) -  \mathcal{A}(\tilde{u}_n, \phi) + \mathcal{A}(\hat{u}, \hat{\phi}_n) = o_n(1)$. Now using the fact that $\tilde{u}_n \rightharpoonup 0$ and $\hat{\phi}_n \rightharpoonup 0$ in $\wps$, applying Lemma \ref{convergence-integrals}-(iii), we get $\mathcal{A}(\tilde{u}_n, \phi) = o_n(1)$ and $\mathcal{A}(\hat{u}, \hat{\phi}_n) = o_n(1)$. Therefore, $\mathcal{A}(\tilde{w}_n, \hat{\phi}_n) = o_n(1)$. Further, for $\tilde{\al} \in [0, sp]$, using Lemma \ref{convergence-BL}-(v), H\"{o}lder's inequality with the conjugate pair $(p^*_s(\tilde{\al}), (p^*_s(\tilde{\al}))')$, the embedding $\wps \hookrightarrow L^{p^*_s(\tilde{\al})}(\rd, \abs{x}^{- \tilde{\al}})$, and $\norm{\hat{\phi}_n}_{\wps} = \norm{\phi}_{\wps}$, we get
\begin{align*}
   \int_{\rd} \frac{\abs{\tilde{w}_n}^{p^*_s(\tilde{\al})-2}\tilde{w}_n}{\abs{x}^{\tilde{\al}}} \hat{\phi}_n \dx - \int_{\rd} \frac{\abs{\hat{u}_n}^{p^*_s(\tilde{\al})-2}\hat{u}_n}{\abs{x}^{\tilde{\al}}}  \hat{\phi}_n \dx
   =\int_{\rd} \frac{\abs{\hat{u}}^{p^*_s(\tilde{\al})-2}\hat{u}}{\abs{x}^{\tilde{\al}}} \hat{\phi}_n \dx + o_n(1).
\end{align*}
Again, the change of variable yields, 
\begin{align*}
    \int_{\rd} \frac{\abs{\tilde{w}_n}^{p^*_s(\tilde{\al})-2}\tilde{w}_n}{\abs{x}^{\tilde{\al}}} \hat{\phi}_n\dx - \int_{\rd} \frac{\abs{\tilde{u}_n}^{p^*_s(\tilde{\al})-2}\tilde{u}_n}{\abs{x}^{\tilde{\al}}} \phi \dx=\int_{\rd} \frac{\abs{\hat{u}}^{p^*_s(\tilde{\al})-2} \hat{u}}{\abs{x}^{\tilde{\al}}} \hat{\phi}_n \dx + o_n(1).
\end{align*}
Now, in view of the above identity, using Lemma \ref{convergence-integrals}-(ii) and the fact that $\hat{\phi}_n \rightharpoonup 0$ in $\wps$, we obtain
\begin{align*}
    \int_{\rd} \frac{\abs{\tilde{w}_n}^{p^*_s(\tilde{\al})-2}\tilde{w}_n}{\abs{x}^{\tilde{\al}}} \hat{\phi}_n\dx = o_n(1). 
\end{align*}
From \eqref{ii-2}, we finally get $\prescript{}{(\wps)'}{\langle} I_{\mu,0,\alpha}'(w_n), \phi {\rangle}_{\wps} = o_n(1)$ for every $\phi \in \cc (\rd)$. Thus, $\{w_n\}$ becomes a (PS) sequence of $I_{\mu,0,\alpha}$ at level $\eta - I_{\mu,a,\alpha}(\tilde{u}) - I_{\mu,0,\alpha}(\hat{u})$.

\noi \textbf{Step 4:} In this step, we consider the case $\al=0$ and $\hat{u} =0$. In view of \eqref{zero norm}, we have $\norm{\phi \hat{u}_n}_{\wps} = o_n(1)$ for $\phi \in \cc(B(0,1))$. The Sobolev embedding $\wps \hookrightarrow L^{p^*_s}(\rd)$ infers that $ \norm{\phi \hat{u}_n}_{L^{p^*_s}(B(0,1))} = o_n(1)$. Since $\phi \in \cc(B(0,1))$ is arbitrary, for any $r\in (0,1)$ we can choose $\phi \equiv 1$ on $B(0,r)$. Therefore, 
\begin{align*}
    \int_{B(0,r)} \abs{\hat{u}_n}^{p^*_s} \dx = o_n(1), \text{ for every } 0<r<1. 
\end{align*}
Hence, in view of the concentration-compactness principle (see \cite[Theorem 1.1]{BSS2018}), there exists a bounded measure $\nu$ such that the following convergence hold in duality with $\mathcal{C}_b(\rd)$:
\begin{align}\label{concentration}
    \abs{\hat{u}_n}^{p^*_s} \chi_{\overline{B(0,1)}}  \overset{\ast}{\rightharpoonup} \nu, \text{ where } \nu=\sum_{i \in I} \nu_i \delta_{x_i}, x_i \in \rd \text{ satisfies } |x_i| = 1, \text{ and } \nu_i = \nu(\{ x_i\}).
\end{align}
Further, 
\begin{align*}
   \limsup_{n \ra \infty} \int_{\rd} \abs{\hat{u}_n}^{p^*_s} \chi_{\overline{B(0,1)}} \dx = \nu(\rd),
\end{align*}
as $\nu_{\infty} = 0$. Now, since $\{ \hat{u}_n \}$ is bounded in $L^{p^*_s}(\rd)$, from \eqref{concentration} and the above convergence, the index set $I$ is finite.  
Let $M= \max\{ \nu_i : i \in I\}$. Then $M< \infty$. Now, we define the Levy concentration function
\begin{align*}
    P_n(r) := \sup_{y \in \rd} \int_{B(y,r)} \abs{ \hat{u}_n }^{p^*_s} \dx.
\end{align*}
Note that, $P_n$ is continuous on $\R^+$ (see \cite[Lemma 3.1]{Brasco-2016}).
Take $\phi \in \mathcal{C}_b(\rd)$ with $\phi \equiv 1$ in $\overline{B(0,1)}$ and $\phi \equiv 0$ in $\rd \setminus B(0,2)$. In view of \eqref{concentration},
\begin{align*}
    \int_{B(0,1)} \abs{\hat{u}_n}^{p^*_s} \dx = \int_{\rd} \abs{\hat{u}_n}^{p^*_s} \chi_{\overline{B(0,1)}} \phi(x) \dx = \sum_{i \in I} \nu_i \phi(x_i) + o_n(1) = \sum_{i \in I} \nu_i + o_n(1).
\end{align*}
Therefore, there exists $\tau \in (0,1)$ such that $P_n(\infty) > M \tau$ and $ P_n(r) > M \tau$ for each $r>0$ large enough.
Further, using \eqref{concentration}, for every $r>0$, $\liminf_{n \ra \infty} P_n(r) \ge M\tau$. Also, $P_n(0)< M \tau$. 
These yield the existence of $\{s_n\} \subset \R^+$ and $\{y_n\} \subset \rd$ with $s_n \ra 0$ and $\abs{y_n}> \frac{1}{2}$ such that 
\begin{align}\label{L-1}
    M \tau = P_n(s_n) = \int_{B(y_n,s_n)} \abs{ \hat{u}_n }^{p^*_s} \dx.
\end{align}
Define 
\begin{align*}
    \hat{v}_n(z) := s_n^{\frac{d-sp}{p}} \hat{u}_n(s_n z + y_n), \text{ for } z \in \rd. 
\end{align*}
Observe that $\norm{ \hat{v}_n }_{\wps} = \norm{ \hat{u}_n }_{\wps}$. Hence the sequence $\{\hat{v}_n\}$ is bounded in $\wps$. By the reflexivity, $\hat{v}_n \rightharpoonup \hat{v}$ in $\wps$. If $\hat{v} =0$, then using similar set of arguments we can show that $\norm{\phi \hat{v}_n}_{\wps} = o_n(1)$ for every $\phi \in \cc(B(0,1))$, and then the Sobolev inequality yields 
\begin{align*}
    \int_{B(0,r)} \abs{\hat{v}_n}^{p^*_s} \dx = o_n(1), \text{ for every } 0<r<1. 
\end{align*}
On the other hand, in view of \eqref{L-1} we see that 
\begin{align*}
    \int_{B(0,1)} \abs{\hat{v}_n}^{p^*_s} \dx = M \tau,
\end{align*}
a contradiction. Thus, $\hat{v} \neq 0$. 
Define $R_n = r_n s_n$ and $z_n = r_n y_n$. Notice that
\begin{align*}
    \hat{v}_n(z) = R_n^{\frac{d-sp}{p}} \tilde{u}_n(R_n z + z_n), \text{ for } z \in \frac{\Omega - z_n}{R_n},
\end{align*}
where $R_n = o_n(1)$, $z_n \ra z_0 \in \rd$ or $\abs{z_n} \ra \infty$, and $\frac{R_n}{|z_n|} = \frac{s_n}{|y_n|} < 2 s_n =o_n(1)$.  
As a consequence, $\left| \rd \setminus \frac{\Omega - z_n}{R_n} \right| \ra 0$, as $n \ra \infty$.  

Now, we show that the non-zero weak limit $\hat{v}$ weakly solves the following limiting equation
\begin{align}\label{limiting problem-2}
    (-\Delta_p)^s u=|u|^{p^*_s-2}u \;\mbox{ in }\,\mathbb{R}^d. 
\end{align}
Take $\psi \in \wps$. For $n \in \mathbb{N}$, we set 
\begin{align*}
    \psi_n(z) = R_n^{- \frac{d-sp}{p}} \psi \left(\frac{z-z_n}{R_n} \right), \text{ for } z \in \rd.
\end{align*}
Observe that $\norm{\psi_n}_{\wps} = \norm{\psi}_{\wps}$. Using the change of variable $\overline{x}_n = R_nx+z_n,\,\overline{y}_n = R_n y+z_n$, we similarly get (as in Step 3), $\mathcal{A}(\hat{v}_n, \psi) = \mathcal{A}(\tilde{u}_n, \psi_n)$, and 
\begin{align*}
    \int_{\rd} \abs{\hat{v}_n}^{p^*_s-2}\hat{v}_n \psi \dx = \int_{\rd} \abs{\tilde{u}_n}^{p^*_s-2}\tilde{u}_n \psi_n \dx. 
\end{align*}
Hence using $\prescript{}{(\wps)'}{\langle} I_{\mu,0,0}'(\tilde{u}_n), \psi_n {\rangle}_{\wps} \ra 0$,
\begin{align}\label{estimate-1}
    &\mathcal{A}(\hat{v}_n, \psi) = \mu \int_{\Omega} \frac{\abs{\tilde{u}_n}^{p-2} \tilde{u}_n}{\abs{x}^{sp}} \psi_n \dx - \int_{\Omega} a(x) \abs{\tilde{u}_n}^{p-2} \tilde{u}_n \psi_n \dx + \int_{\rd} \abs{\hat{v}_n}^{p^*_s-2}\hat{v}_n \psi \dx + o_n(1).
\end{align}
Now using the change of variable $\overline{x}_n = \frac{x-z_n}{R_n}$, we see that  
\begin{align*}
    \int_{\rd} \frac{\abs{\tilde{u}_n}^{p-2} \tilde{u}_n}{\abs{x}^{sp}} \psi_n \dx = \frac{R_n^{\frac{dp-d+sp}{p}}}{R_n^{sp}} \int_{\rd} \frac{\abs{\tilde{u}_n}^{p-2} \tilde{u}_n}{\abs{ \overline{x}_n +\frac{z_n}{R_n}}^{sp}} \psi \d \overline{x}_n = \int_{\rd} \frac{\abs{\hat{v}_n}^{p-2} \hat{v}_n}{\abs{x + \frac{z_n}{R_n}}^{sp}} \psi \dx.
\end{align*}
Since $\frac{\abs{z_n}}{R_n} \ra \infty$, there exists $n_0 \in \N$ such that for all $n \ge n_0$, $\abs{x+\frac{z_n}{R_n}} \ge \abs{x}$, and hence  
\begin{align*}
   \left| \int_{\rd} \frac{\abs{\hat{v}_n}^{p-2} \hat{v}_n}{\abs{x + \frac{z_n}{R_n}}^{sp}} \psi \dx \right| \le \int_{\rd} \frac{\abs{\hat{v}_n}^{p-1}}{\abs{x + \frac{z_n}{R_n}}^{s(p-1)}} \frac{\abs{\psi}}{\abs{x}^s} \dx = o_n(1),
\end{align*}
since $\frac{\abs{\hat{v}_n}^{p-1}}{\abs{x + \frac{z_n}{R_n}}^{s(p-1)}} \rightharpoonup 0$ in $L^{p'}(\rd)$ and $\frac{\abs{\psi}}{\abs{x}^s} \in L^p(\rd)$ (using \eqref{HS}). Therefore, 
\begin{align*}
    \int_{\rd} \frac{\abs{\tilde{u}_n}^{p-2} \tilde{u}_n}{\abs{x}^{sp}} \psi_n \dx = o_n(1), \text{ for every } \psi \in \wps.
\end{align*} 
Now, taking the limit as $n \ra \infty$ in \eqref{estimate-1}, and using Lemma \ref{compact-embedding} and Lemma \ref{convergence-integrals}, we get that $\hat{v}$ weakly solves \eqref{limiting problem-2}. 

Next, we set 
\begin{align*}
    h_n(z) = \tilde{u}_n(z) - R_n^{-\frac{d-sp}{p}} \hat{v} \left( \frac{z-z_n}{R_n} \right) \text{ and } \tilde{h}_n(z) = R_n^{\frac{d-sp}{p}} h_n(R_nz + z_n), \text{ for } z \in \rd.
\end{align*}
Note that $\norm{h_n}_{\wps} = \norm{\tilde{h}_n}_{\wps}$ and $\tilde{h}_n = \hat{v}_n - \hat{v}$. Hence the norm invariance gives $\norm{ h_n }_{\wps} = \norm{ \tilde{h}_n }_{\wps} = \norm{ \hat{v}_n - \hat{v} }_{\wps}$.  Applying Lemma \ref{convergence-BL}, we see that 
\begin{align*}
    \norm{ h_n }_{\wps}^p = \norm{ \hat{v}_n }_{\wps}^p - \norm{ \hat{v} }_{\wps}^p + o_n(1) = \norm{ \tilde{u}_n }_{\wps}^p - \norm{ \hat{v} }_{\wps}^p + o_n(1).
\end{align*}
We show that $\{h_n\}$ is a (PS) sequence of $I_{\mu,0,0}$ at level $\eta - I_{\mu,a,\alpha}(\tilde{u}) - I_{0,0,0}(\hat{v})$. 
Applying Lemma \ref{convergence-BL}, and the fact that $I_{\mu,0,0}(\tilde{u}_n) = \eta -I_{\mu,a,\alpha}(\tilde{u}) + o_n(1)$, we see that 
\begin{align*}
    I_{\mu,0,0} (h_n) & = \frac{1}{p}\norm{h_n}_{\wps}^p - \frac{\mu}{p} \int_{\rd}\frac{\abs{h_n}^p}{\abs{x}^{sp}} \dx - \frac{1}{p^*_s} \int_{\rd} \abs{h_n}^{p^*_s} \dx \\
    & = \frac{1}{p} \left( \norm{\hat{v}_n}_{\wps}^p - \norm{\hat{v}}_{\wps}^p \right) - \frac{\mu}{p} \int_{\rd}\frac{\abs{h_n}^p}{\abs{x}^{sp}} \dx - \frac{1}{p^*_s} \left( \int_{\rd} \abs{\hat{v}_n}^{p^*_s} \dx - \int_{\rd} \abs{\hat{v}}^{p^*_s} \dx \right) + o_n(1) \\
    & = \frac{1}{p} \left( \norm{\tilde{u}_n}_{\wps}^p - \norm{\hat{v}}_{\wps}^p \right) - \frac{\mu}{p} \int_{\rd}\frac{\abs{\tilde{u}_n}^p}{\abs{x}^{sp}} \dx - \frac{1}{p^*_s} \left( \int_{\rd} \abs{\tilde{u}_n}^{p^*_s} \dx - \int_{\rd} \abs{\hat{v}}^{p^*_s} \dx \right) + o_n(1) \\
    &= I_{\mu,0,0}(\tilde{u}_n) - I_{0,0,0}(\hat{v}) + o_n(1) = \eta - I_{\mu,a,\alpha}(\tilde{u}) - I_{0,0,0}(\hat{v}) + o_n(1),
\end{align*}
where the third equality follows using Lemma \ref{convergence-BL}-(ii) and the fact that 
\begin{align*}
   \int_{\rd}\frac{\abs{h_n - \tilde{u}_n}^p}{\abs{x}^{sp}} \dx = \int_{\rd}  \frac{\abs{\hat{v}}^p}{\abs{x+\frac{z_n}{R_n}}^{sp}} \dx = o_n(1), \text{ since }  \frac{\abs{z_n}}{R_n} \ra \infty.
\end{align*}
Moreover, using a similar set of arguments as in Step 3, we get $\prescript{}{(\wps)'}{\langle} I_{\mu,0,0}'(h_n), \phi {\rangle}_{\wps} = o_n(1)$ for every $\phi \in \wps$.
Thus $\{h_n\}$ becomes a (PS) sequence of $I_{\mu,0,0}$ at level $\eta - I_{\mu,a,\alpha}(\tilde{u}) - I_{0,0,0}(\hat{v})$. 

\noi \textbf{Step 5:} Starting from a (PS) sequence $\{ \tilde{u}_n \}$ of $I_{\mu,0,\alpha}$ we have extracted further (PS) sequences at a level whose energy is strictly lower than the level of $\{ \tilde{u}_n \}$, and with a fixed amount of decrease in every step, since
\begin{align*}
    I_{\mu,0,\alpha}(\hat{u}) \ge \frac{sp-\al}{p(d-\al)} S_{\mu}^{\frac{d-\al}{sp-\al}} \text{ and } I_{0,0,0}(\hat{v}) \ge \frac{s}{d} S^{\frac{d}{sp}}.
\end{align*}
Since $\sup_n\|\tilde u_n\|_{\wps}$ is finite, there exist $n_1,n_2 \in \N$ such that this process terminates after $n_1+n_2$ steps and the last (PS) sequence strongly converges to $0$. Let $\tilde{u}_1$ and $\tilde{u}_2$ be two non-zero weak limits appearing from two different (PS) sequences of distinct levels. Then in the same spirit as \cite{Tintarev} (Page 130, Theorem~3.3) and using \eqref{cov-1} and \eqref{cov-2} respectively, we get 
\begin{align*}
    &\mathcal{A}(C_{x^1_n,R^1_n} \tilde{u}_1,C_{x^2_n,R^2_n} \tilde{u}_2) = \mathcal{A}\left(\tilde{u}_1, C_{\frac{x^2_n-x^1_n}{R_n^1}, \frac{R_n^2}{R_n^1}} \tilde{u}_2 \right) \ra 0,\mbox{ as }n\to\infty, \text{ and } \\
    & \mathcal{A}(C_{r^1_n} \tilde{u}_1,C_{r^2_n} \tilde{u}_2) = \mathcal{A}\left(\tilde{u}_1, C_{\frac{r_n^2}{r_n^1}} \tilde{u}_2 \right) \ra 0,\mbox{ as }n\to\infty
\end{align*}
Hence, in view of Proposition \ref{weak-bub} and Proposition \ref{weak-bub_1}, we obtain
\begin{align*}
    \bigg|\log\left(\frac{R^1_n}{R_n^2}\right)\bigg|+\bigg|\frac{x_n^1-x_n^2}{R_n^1}\bigg|\to\infty \text{ and } \bigg|\log\left(\frac{r^1_n}{r_n^2}\right)\bigg| \to\infty,\text{ as }n\to\infty.
\end{align*}
This completes the proof. \qed

As an application of Theorem \ref{PS-decomposition}, we have the following remarks. 

\begin{remark}
  Let $\{u_n\}$ be a (PS) sequence for $I_{\mu,a, \alpha}$ with $\norm{(u_n)^-}_{L^{p^*_s}(\Omega)} \ra 0$ as $n \ra \infty$. Then Theorem \ref{PS-decomposition} holds with $u \ge 0$ a.e. in $\Omega$, $\tilde{u}_i \ge 0$ and $\tilde{U}_j \ge 0$ a.e. in $\rd$. 
  \end{remark}

\begin{remark}[Constrained minimization problem]
Consider the Nehari manifold associated with \eqref{MainEq}, $\mathcal{N}:= \left\{ u \in \wos : \prescript{}{(\wps)'}{\langle} I'_{\mu, a, \alpha}(u),u{\rangle}_{\wps} = 0 \right\}$. Suppose
\begin{align}\label{lower-level}
    l:= \inf_{u \in \mathcal{N}} < \min \left\{ \frac{s}{d} S^{\frac{d}{sp}}, \frac{sp-\al}{p(d-\al)} S_{\mu}^{\frac{d-\al}{sp-\al}} \right\}. 
\end{align}
By the Ekeland variational principle, the functional $I_{\mu, a, \alpha}$ restricted to $\mathcal{N}$ has a (PS) sequence at level $l$, and in view of \eqref{lower-level}, Theorem \ref{PS-decomposition} infers that $\{ u_n \}$ contains a subsequence which converges to a minimizer of $l$, and this minimizer weakly solves \eqref{MainEq}.  
\end{remark}

\noi \textbf{Acknowledgments:}
The author acknowledges the support of the National Board for Higher Mathematics (NBHM) Postdoctoral Fellowship (0204/16(9)/2024/RD-II/6761). The author thanks Souptik Chakraborty for the valuable discussion.

\noi \textbf{Declaration:} 
The author declares that no data were generated or analysed during the current study. As such, there is no data availability associated with this manuscript.

\bibliographystyle{abbrvnat}

\end{document}